\documentclass[11pt]{amsart}
\usepackage{extsizes}
\usepackage{fullpage}
\usepackage{amsfonts,xcolor,url,verbatim}
\usepackage{amssymb}
\usepackage{xcolor}
\usepackage{dsfont}

\usepackage{todonotes}
\usepackage{hyperref}

\newtheorem{thm}{Theorem}[section]
\newtheorem{lemma}[thm]{Lemma}
\newtheorem{prop}[thm]{Proposition}

\newtheorem{remark}[thm]{Remark}

\numberwithin{equation}{section}

\renewcommand{\Re}{{\mathfrak{Re}}}

\newcommand{\authorfootnotes}{\renewcommand\thefootnote{\@fnsymbol\c@footnote}}%
\makeatother

\begin{document}

\def\q{\mathfrak{q}}
\def\p{\mathfrak{p}}
\def\u{\mathfrak{u}}
\def\a{\mathfrak{a}}
\def\b{\mathfrak{b}}
\def\m{\mathfrak{m}}
\def\n{\mathfrak{n}}
\def\c{\mathfrak{c}}
\def\d{\mathfrak{d}}
\def\e{\mathfrak{e}}
\def\k{\mathfrak{k}}
\def\z{\mathfrak{z}}
\def\h{{\mathfrak h}}
\def\gl{\mathfrak{gl}}
\def\sl{\mathfrak{sl}}
\def\t{\mathfrak{t}}

\def\Ext{{\rm Ext}}
\def\Hom{{\rm Hom}}
\def\Ind{{\rm Ind}}

\def\res{\mathop{Res}}

\def\GL{{\rm GL}}
\def\SL{{\rm SL}}
\def\SO{{\rm SO}}
\def\O{{\rm O}}

\def\R{\mathbb{R}}
\def\C{\mathbb{C}}
\def\Z{\mathbb{Z}}
\def\Q{\mathbb{Q}}
\def\A{\mathbb{A}}
\def\bk{k}

\def\w{\wedge}

\def\Cat{\mathcal{C}}
\def\HC{{\rm HC}}
\def\HCat{\Cat^\HC}
\def\proj{{\rm proj}}

\def\to{\rightarrow}
\def\To{\longrightarrow}

\def\1{1\!\!1}
\def\dim{{\rm dim}}

\def\th{^{\rm th}}
\def\isom{\approx}

\def\CE{\mathcal{C}\mathcal{E}}
\def\E{\mathcal{E}}

\def\dis{\displaystyle}
\def\f{{\bf f}}                 
\def\g{{\bf g}}
\def\T{{\rm T}}              
\def\omegatil{\tilde{\omega}}  
\def\H{\mathcal{H}}            
\def\W{W^{\circ}}           
\def\Whit{\mathcal{W}}      
\def\ringO{\mathcal{O}}     
\def\S{\mathcal{S}}      
\def\M{\mathcal{M}}      
\def\K{{\rm K}}          
\def\h{\mathfrak{h}} 
\def\N{\mathfrak{N}}    
\def\norm{{\rm N}}       
\def\trace{{\rm Tr}} 
\def\ctilde{\tilde{C}}

\def\Sym{{\rm Sym}} 
\def\Ad{{\rm Ad}} 
\def\Gal{{\rm Gal}} 

\def\ord{{\rm ord}}

\title{Low-Lying Zeros of $L$-functions of Ad\'elic Hilbert Modular Forms and their Convolutions}

\author{Alia Hamieh}
\address[Alia Hamieh]{University of Northern British Columbia\\ Department of Mathematics and Statistics\\ Prince George, BC V2N 4Z9 Canada}
\email{alia.hamieh@unbc.ca}

\author{Peng-Jie Wong}
\address[Peng-Jie Wong]{National Sun Yat-Sen University\\Department of Applied Mathematics\\
Kaohsiung City, Taiwan}
\email{pjwong@math.nsysu.edu.tw}


\keywords{Hilbert modular forms, Rankin-Selberg convolutions, 1-level density, automorphic $L$-functions, central values}

\subjclass[2010]{Primary 11F41, 11F67; secondary 11F30, 11F11, 11F12, 11N75}
\date{\today}

\begin{abstract}
In this article, we study the density conjecture of Katz and Sarnak for $L$-functions of ad\'elic Hilbert modular forms and their convolutions. In particular, under the generalised Riemann hypothesis, we establish several instances supporting the conjecture and extending the works of Iwaniec-Luo-Sarnak and many others. For applications, we obtain an upper bound for the average order of $L$-functions of Hilbert modular forms at $s=\frac{1}{2}$ as well as a positive proportion of non-vanishing of certain Rankin-Selberg $L$-functions.
 \end{abstract}

\maketitle

\section{Introduction}

Since the seminal discovery by Montgomery and Dyson, random matrices have been playing an essential role in helping number theorists understand the behaviour of zeros
of $L$-functions. Profoundly, a conjecture of Katz and Sarnak \cite{katz-sarnak}, which is often called the density conjecture, predicts that the distribution of ``low-lying'' zeros of $L$-functions in a family $\mathfrak{F}$ is determined by the ``symmetry type'' of $\mathfrak{F}$. (We refer the reader to  \cite[Sec. 1.1]{liu-miller} for a precise formulation for the density conjecture.)

Several results tested and supported the conjecture for various families, including $L$-functions of modular forms and their symmetric lifts, under the generalised Riemann hypothesis (GRH). For instance, in their monumental article \cite{Iwaniec-Luo-Sarnak}, Iwaniec, Luo, and Sarnak studied the above-mentioned density conjecture of Katz and Sarnak for holomorphic modular forms and their symmetric squares. Nowadays, their work has a variety of extensions. For example, the condition on square-free levels was removed recently in \cite{miller-arbitrarylevel}. Also, the cases of Hilbert modular forms (over totally real number fields of narrow class number one) and automorphic forms on totally definite quaternion algebras (including both  Maass forms and Hilbert modular forms over general totally real number fields) have been established by Liu-Miller \cite{liu-miller} and Lesesvre \cite{Lesesvre-QA}, receptively. Recently, extending the previous work of  Ricotta and Royer \cite{Ricotta-Royer} on low-lying zeros of symmetric power $L$-functions of holomorphic modular forms, in the level aspect, Sugiyama \cite{Sugiyama} further verified the density conjecture, weighted by symmetric square $L$-values, for symmetric power $L$-functions of Hilbert modular forms (under some mild conditions on underlying number fields). More generally, Shin and Templier \cite{shin-templier} proved the density conjecture for a wide family of automorphic representations. In addition, Shashkov \cite{shashkov} recently proved several results that support the density conjecture for families of Rankin-Selberg $L$-functions of holomorphic modular forms over $\Bbb{Q}$.

Recall that the 1-level density of an $L$-function $L(s, f)$, as formalised in \cite[Ch. 5]{IK} (see also Section \ref{AF} below), is defined by 
$$
D(f;\phi) =\sum_{\gamma_{f}} \phi \left( \frac{\gamma_{f}}{2\pi} \log c_{f}\right),
$$
where $\phi(x)$ is an even Schwartz function, the sum is over the zeros $\frac{1}{2} +i\gamma_{f}$ of $f$, and $c_f$ stands for the analytic conductor of $f$.
Our first result in this paper generalises one of the theorems of Iwaniec, Luo, and Sarnak \cite{Iwaniec-Luo-Sarnak} (as well as the above-mentioned result of Liu-Miller \cite{liu-miller}) to holomorphic Hilbert modular forms over \emph{arbitrary} totally real number fields.

\begin{thm}\label{thm:main1}
Let $F$ be a totally real number field of degree $n$. Let $k=(k_1,\dots, k_n)\in 2\mathbb{N}^{n}$ and $\n$ denote a square-free integral ideal of $F$. We denote by $\Pi_{k}(\n)$ the set of all primitive forms of weight $k$ and level $\n$. Assume GRH for $\zeta_F(s)$, $L(s,\f)$, and $L(s,\Sym^2(\f))$ for all $\f\in \Pi_{k}(\n)$. Then for any Schwartz class function $\phi(x)$  whose Fourier transform $\widehat{\phi}$ is supported in $(-u, u)$ with 
$$
u = \frac32\frac{\log\norm(\n)}{\log\left(\norm(\n)\norm(k)^2\right)} +\frac43 \frac{\log \norm(k)}{ \log (\norm(\n)\norm(k)^2)},
$$ 
where $\norm(k)=\prod_j k_j$, 
we have
$$
  \frac{1}{|\Pi_{k}(\n)|}
\sum_{\f\in \Pi_{k}(\n)} D(\f;\phi)
\sim\int_{-\infty}^{\infty} \phi(x) W( \mathrm{O})(x) dx,
$$
as  $\norm(k)\norm(\n)\rightarrow \infty$, 
where ${W} (\mathrm{O})(x) = 1+ \frac{1}{2}\delta_0(x)$ and $\delta_0$ is the standard Dirac delta functional. 
\end{thm}

We note that if $\log(\norm(k))=o\left(\log(\norm(\n))\right)$, the choice $u=\frac{3}{2}$ is admissible. This improves both the main theorem of \cite{Lesesvre-QA} and the special instance of \cite[Theorem 1.2]{Sugiyama} (with $r=z=1$) in the level aspect. If $\norm(k)\rightarrow \infty$ while $\norm(\n)$ is bounded or fixed, one can only take $u=\frac23$. The family considered in Theorem \ref{thm:main1} is also included in \cite{shin-templier}. However, as remarked in \cite{Lesesvre-QA}, the limit of the supports of the Fourier transforms was not provided explicitly in \cite{shin-templier}, and this is in contrast with all the above-mentioned results in which explicit limits of supports are given.  One of the main features of our result is to break the $(-1,1)$-barrier. In addition, we remove both the class-number-one condition and the parallel-weight condition in \cite{liu-miller} while keeping $\widehat{\phi}$ supported in $(-\frac{3}{2}, \frac{3}{2})$ for Hilbert modular forms of large level.

Our second result in this paper provides a 1-level density for a family of Rankin-Selberg $L$-functions of Hilbert modular forms over totally real number fields.
\begin{thm} \label{thm:main2}
Let $F$ be a totally real number field. Let $k,k'\in 2\mathbb{N}^{n}$ and $\n,\n'$ denote coprime square-free integral ideals of $F$. Fix  $\g\in \Pi_{k'}(\n')$, and  assume GRH for $\zeta_F(s)$, $L(s,\f\times \g)$, and $L(s,\Sym^2(\f\times \g))$ for all $\f\in \Pi_{k}(\n)$. If either the class number of $F$ is odd or $\g$ is non-dihedral, then for any Schwartz class function $\phi(x)$ whose Fourier transform $\widehat{\phi}$ is supported in  $(-u, u)$ with 
$$
u = \frac34\frac{\log\norm(\n)}{\log\left(\norm(\n)\norm(k)^2\right)} +\frac23 \frac{\log \norm(k)}{ \log (\norm(\n)\norm(k)^2)},
$$ 
as $\norm(k)\norm(\n)\rightarrow \infty$, we have
$$
 \frac{1}{|\Pi_{k}(\n)|}
\sum_{f\in \Pi_{k}(\n)} D(f\times g;\phi)
\sim\int_{-\infty}^{\infty} \phi(x) W( \mathrm{Sp})(x) dx,
$$
where
$
W( \mathrm{Sp})(x) = 1 -\frac{\sin (2\pi x)}{2\pi x}.
$
\end{thm} 

\begin{remark}
It is worth noting that to establish the above theorem, we require the precise information of the order of the pole of $L(s, \Sym^2 ( \f\times \g) )$ at $s=1$ (denoted by $\delta_{\f\times\g}$) to calculate the averaged asymptotics of $\delta_{\f\times \g}$ in Section \ref{sec-averaged-asymp}. Although the evaluation of  $\delta_{\f\times\g}$ is known by experts (at least, implicitly), it is seemingly hard to find a general reference. Therefore, we will devote Section \ref{delta-fxg} to calculating $\delta_{\f\times\g}$ and record the precise values of $\delta_{\f\times\g}$ in  Theorem \ref{delta_fg}.
\end{remark}

%
%
%
%
%
%

As mentioned earlier, Shashkov \cite{shashkov} recently obtained a version of Theorem \ref{thm:main2} for Rankin-Selberg $L$-functions of holomorphic modular forms. In fact, denoting $H^{\star}_k (N)$ the set of holomorphic newforms of weight $k$ and level $N$, he showed that
$$
\lim_{NN' \rightarrow \infty}  \frac{1}{|H^{\star}_{k}(N)||H^{\star}_{k'}(N')|}
\sum_{f\in H^{\star}_{k}(N)}\sum_{g\in H^{\star}_{k'}(N')} D(f\times g;\phi)
=\int_{-\infty}^{\infty} \phi(x) W( \mathrm{Sp})(x) dx,
$$
with $\widehat{\phi}$ supported in $(-\frac{1}{2}, \frac{1}{2})$. In addition, the support can be improved to $(-\frac{5}{4}, \frac{5}{4})$ if $N=N'$ but $k\neq k'$, and be taken as $(-\frac{29}{28}, \frac{29}{28})$ if  $N=N'$ and $k= k'$. Compared to Theorem \ref{thm:main2}, Shashkov considered a double-averaging. It may be possible to extend Shashkov's argument to convolutions of Hilbert modular forms (at least, over totally real number fields of  class number one). However, we did not pursue this here as it does not lead to better support for the general case $\n\neq \n'$.

It is well-known that if one could prove the density conjecture for test functions with sufficiently large  support, then one would establish a positive proportion of non-vanishing for the corresponding family  of $L$-functions. Unfortunately, for the family of Hilbert modular forms, one would need to extend the support beyond $(-2, 2)$ to get a positive proportion.\footnote{Note that $(-2, 2)$ is proven by  Iwaniec, Luo, and Sarnak in \cite[Theorem 7.1]{Iwaniec-Luo-Sarnak} via a more dedicated analysis of certain sums of Kloosterman sums in \cite[Sec. 6]{Iwaniec-Luo-Sarnak}.} Instead, we have the following theorem regarding the average order of the zeros of $L(s,\f)$ at $s=\frac{1}{2}$ (which particular improves \cite[Corollary 1.3]{Lesesvre-QA} in the case of Hilbert modular forms). Nevertheless, for Rankin-Selberg convolutions, our result is strong enough to establish a positive proportion of non-vanishing as follows.

\begin{thm}\label{thm:main3} (i) In the notation of Theorem \ref{thm:main1}, for $m\in\Bbb{Z}^+$, we define
\begin{equation}\label{def-Pm}
\mathcal{P}_m(\n) =\frac{1}{|\Pi_{k}(\n)|} \sum_{\f \in \Pi_{k}(\n)} \delta(\f;m),
\quad
\text{with}
\quad
 \delta(\f;m)
=
\begin{cases}
 1 & \text{if $\ord_{s=\frac{1}{2}}L(s,\f) =m$}; \\
 0 & \text{otherwise}.
   \end{cases}
\end{equation}
Then under GRH, we have
\begin{equation}\label{upper-bd-mPm}
\limsup_{\norm(k)\norm(\n)\rightarrow \infty} \sum_{m=1}^\infty m \mathcal{P}_m(\n) \le \frac{1}{u}+\frac{1}{2}.
\end{equation}
(ii) In the notation of Theorem \ref{thm:main2}, for $m\in\Bbb{Z}^+$, we assume GRH and set
\begin{equation}\label{def-Qm}
\mathcal{Q}_m(\n) =\frac{1}{|\Pi_{k}(\n)|} \sum_{\f \in \Pi_{k}(\n)} \delta(\f\times\g;m)
\quad
\text{with}
\quad
 \delta(\f\times\g;m)
=
\begin{cases}
 1 & \text{if $\ord_{s=\frac{1}{2}}L(s,\f \times\g) =m$}; \\
 0 & \text{otherwise}.
   \end{cases}
\end{equation}
If the class number of $F$ is odd or $\g$ is non-dihedral, then 
$$
\limsup_{\norm(k)\norm(\n)\rightarrow \infty} \sum_{m=1}^\infty m \mathcal{Q}_m(\n) \le \frac{1}{2u} -\frac{1}{4} 
$$
Furthermore, (without assuming that either the class number of $F$ is odd or $\g$ is non-dihedral), we have
$$
\liminf_{\norm(k)\norm(\n)\rightarrow \infty} \mathcal{Q}_0(\n) 
\ge\frac{5}{4} -  \frac{1}{2u}. 
$$
\end{thm}

In an upcoming work, the present authors employ ideas from {Radziwi\l\l} and Soundararajan \cite{radziwill-sound} to extend the above non-vanishing result, Theorem \ref{thm:main3}, to a conditional lower bound towards Keating-Snaith's normality conjecture on central $L$-values.

\section{A general explicit formula}

In this section, we shall consider a formalism of $L$-functions over a number field $F$ in much the same spirit as \cite[Ch. 5]{IK}, and derive a general explicit formula for such $L$-functions.
\subsection{An analytic formalism}\label{AF}

Let $F$ be a number field of degree $n$ over $\mathbb{Q}$. We denote its ring of integers by $\ringO_F$. For any given symbol $f$ and a complex variable $s$, $L(s,f)$ is said to be an $L$-function of degree $d$ over $F$ if it satisfies conditions (i) and (ii) below.

\noindent(i) The function  $L(s,f)$ admits a Dirichlet series with an Euler product of degree $d\ge 1$, i.e.,
$$
L(s,f)
=\sum_{\n\neq 0 } \lambda_{f}(\n)\norm(\n)^{-s}
=\prod_{\p}\prod_{i=1}^{d}(1-\alpha_{i}(\p)\norm(\p)^{-s})^{-1},
$$
which are absolutely convergent on $\Re(s)>1$. Here, the sum is over non-zero integral ideals $\n\subseteq \ringO_{F}$, and the product is over prime ideals of $F$. In addition, $\lambda_{f}(\ringO_{F})=1$, $\lambda_{f}(\n)\in\Bbb{C}$, $\alpha_{i}(\p)\in\Bbb{C}$, and $|\alpha_{i}(\p)|< \norm(\p)$. Moreover, there is an integral ideal $\mathcal{C}_f\subseteq \ringO_{F}$ associated with $f$, called the conductor of $L(s,f)$, such that $\alpha_{i}(\p)\neq 0$ for $\p\nmid\mathcal{C}_f$ and $1\le i \le d$. We will call primes $\p\nmid \mathcal{C}_f$ unramified. 

\noindent(ii) Consider $$
\gamma(s,f)=\pi^{{-dn s}/{2}}\prod_{j=1}^{d n}\Gamma\left(\frac{s+\kappa_{j}}{2}\right),
$$
where $\kappa_{j}\in\Bbb C$ are either real or appear in conjugate pairs, and $\Re(\kappa_{j})>-1$. There is an integer $A_f\ge 1$, associated with $f$, such that the complete $L$-function of $f$ defined by
$$
\Lambda(s,f)=A_f^{s/2}\gamma(s,f)L(s,f),
$$
is a holomorphic function on $\Re(s)>1$. In addition, $\Lambda(s,f)$ extends to a meromorphic function over $\Bbb{C}$ of order $1$, with at most two poles at $s=0$ and $s=1$. Also, it satisfies the functional equation
$$
\Lambda(s,f)=W(f)\Lambda(1-s,\bar{f}),
$$
where $\bar{f}$ is the dual symbol of $f$ for which $\lambda_{\bar{f}}(\n)=\overline{\lambda_{f}(\n)}$, $\gamma(s,\bar{f})=\gamma(s,f)$,  $\mathcal{C}_{\bar{f}}=\mathcal{C}_f$, $A_{\bar{f}}=A_f$, and $W(f)$ is a complex number of absolute value $1$.

The von-Mangoldt function $\Lambda_{f}(\n)$ for $L(s,f)$ is defined by
$$
-\frac{L'(s,f)}{L(s,f)}=\sum_{\n\neq 0}\frac{\Lambda_{f}(\n)}{\norm(\n)^{s}}.
$$ 
Note that $\Lambda_{f}$ is only supported on prime powers, and  for $\n=\p^m$, we have
$$
\Lambda_{f}(\p^m) =\log(\norm(\p)) \sum_{i=1}^d \alpha_i(\p)^m.
$$

Throughout our discussion, we will further assume  the Ramanujan-Petersson conjecture (denoted RPC) for $L(s,f)$ as follows.  For all $\p\nmid \mathcal{C}_f$, $|\alpha_{i}(\p)|=1$; otherwise, $|\alpha_{i}(\p)|\le 1$. Moreover, we shall assume that the asymptotic formula
$$
\sum_{\norm(\p)\le x} \Lambda_f(\mathfrak{p}^2) = \delta_f  x + O(x^{1/2}  \log^2( c_f  x ))
$$
holds for some constants $\delta_f\in \Bbb{Z}$ and $c_f>0$ associated with $f$.

\subsection{Explicit formula}\label{EF}
In this section, we shall further assume that $f$ is self-dual (i.e. $\bar{f}=f$).
Let $h$ be an even Schwartz function with Fourier transform $\widehat{h}$.  By following the proofs of \cite[Theorems~ 5.11 and 5.12]{IK}, we have
\begin{align}\label{eqn:explicit-formula-1}
\sum_{\rho}{h}\left(\frac{\gamma}{2\pi}\right)
&=\widehat{h}(0)\log A_f+\frac{1}{\pi}\int_{-\infty}^{\infty}\frac{\gamma'}{\gamma}(\frac12+it,f){h}\left(\frac{t}{2\pi}\right)\;dt-2\sum_{\mathfrak{n}\neq 0}\Lambda_f(\mathfrak{n})\frac{\widehat{h}(\log \norm(\mathfrak{n}))}{\sqrt{\norm(\mathfrak{n})}},
\end{align}
where the sum is over all zeros $\rho=\beta +i\gamma$ of $L(s,f)$ such that $\beta\in[0,1]$ (including the non-trivial and the trivial ones) counted with multiplicity. 

 Let $R$ be a sufficiently large positive parameter. Observing that the last sum can be written as
$$
\sum_{\mathfrak{p}}\sum_{m\geq1}\Lambda_f(\mathfrak{p}^m)\frac{\widehat{h}(m\log\norm(\mathfrak{p}))}{\norm(\mathfrak{p})^{\frac{m}{2}}},
$$
and applying \eqref{eqn:explicit-formula-1} with $h(x)=\phi(x\log R)$ and $\widehat{h}(y)=\frac{1}{\log R}\widehat{\phi}(\frac{y}{\log R})$ for an even Schwartz function $\phi$, we obtain
\begin{align}\label{eqn:explicit-formula-2}
\begin{split}
&\sum_{\rho}{\phi}\left(\frac{\gamma}{2\pi}\log R\right)\\
&=\widehat{\phi}(0)\frac{\log A_f}{\log R}+\frac{1}{\pi}\int_{-\infty}^{\infty}\frac{\gamma'}{\gamma}(\frac12+it){\phi}\left(\frac{t}{2\pi}\log R\right)\;dt-\frac{2}{\log R}\sum_{\mathfrak{p}}\sum_{m\geq1}\Lambda_f(\mathfrak{p}^m)\frac{\widehat{\phi}\left(m\frac{\log\norm(\mathfrak{p})}{\log R}\right)}{\norm(\mathfrak{p})^{\frac{m}{2}}}.
\end{split}
\end{align}
Since
\[\frac{\gamma'}{\gamma}(s,f)
= \frac{-dn}{2}\log\pi+ \frac{1}{2}\sum_{j=1}^{dn}\frac{\Gamma'}{\Gamma}\left(\frac{s+\kappa_j}{2}\right),
\]
we have
\begin{align*}
&\frac{1}{\pi}\int_{-\infty}^{\infty}\frac{\gamma'}{\gamma}(\frac12+it){\phi}\left(\frac{t}{2\pi}\log R\right)\;dt\\
&=\frac{2}{\log R}\int_{-\infty}^{\infty}\frac{\gamma'}{\gamma}(\frac12+i2\pi\frac{ t}{\log R})\phi(t)\;dt\\&=\frac{-dn\log\pi}{\log R}\int_{-\infty}^{\infty}\phi(t)\;dt
+\frac{1}{\log R}\sum_{j=1}^{dn}\int_{-\infty}^{\infty}\frac{\Gamma'}{\Gamma}\left(\frac{1+2\kappa_j}{4}+i\pi\frac{ t}{\log R}\right)\phi(t)\;dt
\\
&=\frac{-dn\log\pi}{\log R}\widehat{\phi}(0)+\frac{1}{\log R}\sum_{j=1}^{dn}\int_{0}^{\infty}\left(\frac{\Gamma'}{\Gamma}\left(\frac{1+2\kappa_j}{4}+i\pi\frac{ t}{\log R}\right)+\frac{\Gamma'}{\Gamma}\left(\frac{1+2\kappa_j}{4}-i\pi\frac{ t}{\log R}\right)\right)\phi(t)\;dt.
\end{align*}
Using the following properties of the digamma function
\[\frac{\Gamma'}{\Gamma}(a+bi)+\frac{\Gamma'}{\Gamma}(a-bi)=2\frac{\Gamma'}{\Gamma}(a)+O(a^{-2}b^2),\quad\forall a,b\in\mathbb{R},\;a>0,\]
and
 \[\frac{\Gamma'}{\Gamma}(\alpha+\frac14)=\log\alpha+O(1),\quad \forall\alpha\geq\frac14,\]
we get \begin{align}\label{eqn:digamma-factor}
\frac{1}{\pi}\int_{-\infty}^{\infty}\frac{\gamma'}{\gamma}(\frac12+it){\phi}\left(\frac{t}{2\pi}\log R\right)\;dt&=\frac{-dn\log\pi}{\log R}\widehat{\phi}(0)+\frac{1}{\log R}\sum_{j=1}^{dn}\widehat{\phi}(0) \log\kappa_j  +  O\left(\frac{1}{\log R}\right).
\end{align}
Plugging \eqref{eqn:digamma-factor} in \eqref{eqn:explicit-formula-2}, we obtain
\begin{align*}
\sum_{\rho}{\phi}\left(\frac{\gamma}{2\pi}\log R\right)
&=\frac{\widehat{\phi}(0)}{\log R}\Bigg(-dn\log\pi+\log A_f +\log\prod_{j=1}^{{dn}}\kappa_j\Bigg)\\
&-\frac{2}{\log R}\sum_{\mathfrak{p}}\sum_{m\geq1}\Lambda_f(\mathfrak{p}^m)\frac{\widehat{\phi}\left(m\frac{\log\norm(\mathfrak{p})}{\log R}\right)}{\norm(\mathfrak{p})^{\frac{m}{2}}}+O\left(\frac{1}{\log R}\right).
\end{align*}
We shall now consider the double sum appearing in the above expression. Observing that RPC yields
$$
\sum_{\mathfrak{p}}\sum_{m\geq3}\Lambda_f(\mathfrak{p}^m)\widehat{\phi}\left(m\frac{\log\norm(\mathfrak{p})}{\log R}\right)\norm(\mathfrak{p})^{-\frac{m}{2}} =O(1)
$$
we have
\begin{align*}
\frac{2}{\log R}\sum_{\mathfrak{p}}\sum_{m\geq1}\Lambda_f(\mathfrak{p}^m)\frac{\widehat{\phi}\left(m\frac{\log\norm(\mathfrak{p})}{\log R}\right)}{\norm(\mathfrak{p})^{\frac{m}{2}}}
&=\frac{2}{\log R}\sum_{\mathfrak{p}}\Lambda_f(\mathfrak{p})\frac{\widehat{\phi}\left(\frac{\log\norm(\mathfrak{p})}{\log R}\right)}{\norm(\mathfrak{p})^{\frac{1}{2}}}\\
&+\frac{2}{\log R}\sum_{\mathfrak{p}}\Lambda_f(\mathfrak{p}^2)\frac{\widehat{\phi}\left(2\frac{\log\norm(\mathfrak{p})}{\log R}\right)}{\norm(\mathfrak{p})}
+O\left(\frac{1}{\log R}\right).
\end{align*}
Note that it follows from the assumption
$$
\sum_{\norm(\p)\le x} \Lambda_f(\mathfrak{p}^2) = \delta_f  x + O(x^{1/2}  \log^2( c_f  x ))
$$
and Abel's summation formula that
\begin{align*}
&\frac{2}{\log R}\sum_{\mathfrak{p}}\Lambda_f(\mathfrak{p}^2)\frac{\widehat{\phi}\left(2\frac{\log\norm(\mathfrak{p})}{\log R}\right)}{\norm(\mathfrak{p})}\\
&=    \frac{2 \delta_f}{\log R} \int_1^\infty \widehat{\phi}\left(2\frac{\log t}{\log R}\right) \frac{dt}{t} + \frac{2}{\log R} \int_1^\infty \widehat{\phi}\left(2\frac{\log t}{\log R}\right) \frac{d O(t^{1/2}  \log^2( c_f  t)) }{t}\\
&=\frac{\delta_f}{2}\phi(0)+O\left( \frac{\log\log c_f}{\log R}\right).
\end{align*}
Hence, we arrive at the following proposition.

\begin{prop}\label{prop:general-explicit-formula}
Let $L(s,f)$ be an $L$-function satisfying the previous formalism and 
$$
\sum_{\norm(\p)\le x} \Lambda_f(\mathfrak{p}^2) = \delta_f  x + O(x^{1/2}  \log^2( c_f  x )).
$$
Let $\phi$ be an even Schwartz function and set \[D(f;\phi)=\sum_{\rho}{\phi}(\frac{\gamma}{2\pi}\log R),\] where the sum is over all zeros $\rho=\beta +i\gamma$ of $L(s,f)$ such that $0\le \beta\le 1$. If $f$ is self-dual, then
\begin{align*}
D(f;\phi)
=\frac{\widehat{\phi}(0)}{\log R}\log\Bigg( \pi^{-dn} A_f\prod_{j=1}^{{dn}}\kappa_j\Bigg)
-\frac{\delta_f}{2}\phi(0)
-\frac{2}{\log R}\sum_{\mathfrak{p}}\Lambda_f(\mathfrak{p})\frac{\widehat{\phi}\left(\frac{\log\norm(\mathfrak{p})}{\log R}\right)}{\norm(\mathfrak{p})^{\frac{1}{2}}}
+O\left(\frac{\log\log\left( c_f\right)}{\log R}\right).
\end{align*}
\end{prop}

\section{Rankin-Selberg convolutions}\label{sec:RS}

\subsection{Generality}\label{gen}
In this section, we let $F$ be a general number field. We begin by briefly recalling some facts regarding $L$-functions arising from automorphic representations and their Rankin-Selberg convolutions. For a more detailed discussion on the theory of $L$-functions used in this section, we refer the reader to \cite[Sec. 2]{Ramakrishnan} and the references therein.

 Given an (isobaric) automorphic representation $\pi$  for $\GL_n(\Bbb{A}_F)$, with $n\in\Bbb{N}$, the associated $L$-function $L(s,\pi)$ is defined by
$$
L(s,\pi) 
=\prod_{\p}\prod_{i=1}^{n}(1-\alpha_{\pi,i}(\p)\norm(\p)^{-s})^{-1},
$$
which is absolutely convergent on $\Re(s)>1$. In addition, the Rankin-Selberg $L$-function of two (isobaric) automorphic representation $\pi$ and $\pi'$ for $\GL_n(\Bbb{A}_F)$ and $\GL_{n'}(\Bbb{A}_F)$, respectively, is defined by
$$
L(s,\pi \times\pi' ) =\prod_{\p} L_{\p}(s,\pi\times \pi'),
$$
for $\Re(s)>1$, where
$$
 L_{\p}(s,\pi\times \pi')=\prod_{i=1}^{n}\prod_{j=1}^{n'}(1-\alpha_{\pi,i}(\p)\alpha_{\pi',j}(\p)\norm(\p)^{-s})^{-1}
$$
if $\p$ is unramified primes (for both $\pi$ and $\pi'$). Moreover, the $L$-functions of the exterior and symmetric squares of $\pi$, denoted by $\wedge^2(\pi)$ and $\Sym^2 (\pi)$ respectively, satisfy
$$
L(s,\pi\times\pi) =   L(s,\Sym^2 (\pi))L(s,\wedge^2(\pi))
 =\prod_\p L_\p(s,\Sym^2 (\pi))\prod_\p L_\p(s,\wedge^2(\pi)), 
$$
where for unramified $\p$, we have
$$
L_\p(s,\Sym^2 (\pi))\prod_{1\le i\le j\le n}(1-\alpha_{\pi,i}(\p)\alpha_{\pi,j}(\p)\norm(\p)^{-s})^{-1}
$$
and
$$
 L_\p(s,\wedge^2(\pi))=\prod_{1\le i<j\le n}(1-\alpha_{\pi,i}(\p)\alpha_{\pi,j}(\p)\norm(\p)^{-s})^{-1}.
$$

Now, we let $\f$ and $\g$ be cuspidal automorphic representations for $\GL_2(\Bbb{A}_F)$ with trivial central character.\footnote{This implies that $\f$ and $\g$ are self-dual. Indeed, in general, a (unitary)  cuspidal automorphic representation $\pi$ for $\GL_2(\Bbb{A}_F)$ is self-duality is equivalent to $\pi$ having quadratic central character (which is trivial if $\pi$ is not dihedral).} By the work of Gelbart and Jacquet \cite{Gelbart-Jacquet}, we know that
$$
L(s,\f\times \g \times  \f\times \g ) =
L(s, (\f \boxtimes  \f)\times (\g \boxtimes \g) )
=  L(s, (1 \boxplus \Ad(\f) ) \times (1 \boxplus \Ad(\g) ))
$$ factorises into
$$
 L(s,1) L(s,  \Ad (\f) ) L(s , \Ad (\g) )  L(s,  \Ad (\f) \times  \Ad (\g) ),
$$
where $\boxtimes$ is the functorial tensor product, and  $\boxplus$ denotes the (Langlands) isobaric sum of automorphic representations.\footnote{In fact, Gelbart and Jacquet proved that for self-dual $\f$ and $\g$, one has $\f \boxtimes  \f = 1 \boxplus \Ad(\f)$ and $\g \boxtimes  \g = 1 \boxplus \Ad(\g)$, where the adjoint lifts  $\Ad(\f)$ and $\Ad(\g)$ are automorphic representations for $\GL_3(\Bbb{A}_F)$.} In addition, we have
$$
L(s,\f\times \g \times  \f\times \g )=  L(s, \Sym^2 ( \f\times \g) ) L(s, \wedge^2 ( \f\times \g) ),
$$
where
\begin{align*}
L(s, \Sym^2 ( \f\times \g) )
& =  L(s, \wedge^2(\f) \times \wedge^2(\g) ) L(s,  \Sym^2 (\f) \times  \Sym^2 (\g) )\\
& =L(s, 1 ) L(s,  \Ad (\f) \times  \Ad (\g) )
\end{align*}
and
$$
L(s, \wedge^2 ( \f\times \g) ) = L(s,  \Sym^2 (\f) \times \wedge^2(\g)  ) L(s, \wedge^2(\f) \times \Sym^2 (\g) ) 
=L(s,  \Ad (\f) ) L(s , \Ad (\g) ).
$$
Here, we used the fact that $\Sym^2 (\f) = \Ad (\f) $, $\Sym^2 (\g) = \Ad (\g) $, and  $\wedge^2  (\f) =\wedge^2  (\g)=1 $ as $\f$ and $\g$ have trivial central character. 

Consequently, as both $L(s,  \Ad (\f) )$ and $L(s , \Ad (\g) )$ are holomorphic at $s=1$, under GRH for both $L(s, \Sym^2 ( \f\times \g) )$ and $L(s, \wedge^2 ( \f\times \g) )$, we have
\begin{equation}\label{eqn:pnt-f-g}
\sum_{\norm(\p)\le x} \Lambda_{\f\times\g}(\mathfrak{p}^2) = \delta_{\f\times\g}  x + O(x^{1/2}  \log^2( c_{\f\times\g}  x )),
\end{equation}
where $\delta_{\f\times\g}$ is the order of the pole of $L(s, \Sym^2 ( \f\times \g) )$ at $s=1$ (which is $\ge 1$), and  $c_{\f\times\g}$ is the maximum of conductors of $L(s, \Sym^2 ( \f\times \g) )$ and $L(s, \wedge^2 ( \f\times \g) )$. In order to apply Proposition
\ref{prop:general-explicit-formula} for a family of Rankin-Selberg $L$-functions, we should understand the asymptotic behaviour of $\delta_{\f\times\g}$ on average. This is the content of the following section. 

\subsection{Calculating $\delta_{\f\times\g}$}\label{delta-fxg}

As can be seen in the previous section, in order to use Proposition \ref{prop:fg-explicit-formula}, it is crucial to fully understand $\delta_{\f\times\g}$. In light of Section \ref{gen}, we shall calculate $\delta_{\f\times\g}$ for general $\GL(2)$-forms $\f$ and $\g$, which may be of interest to other applications.

We start by recalling some important properties of $\Ad (\f)$ and  $\Ad (\g)$ as follows.
\begin{itemize}
\item[(i)] If $\f$ is non-dihedral, then Gelbart-Jacquet \cite{Gelbart-Jacquet} showed that $\Ad (\f)$ is cuspidal; otherwise, $\Ad (\f)$ is not cuspidal.

\item[(ii)] By \cite[Theorem 4.1.2]{Ramakrishnan}, it is known that  $\Ad (\f) \simeq \Ad (\g) $ if and only if $\f$ and $\g$ are twist-equivalent. 

\item[(iii)] When both $\f$ and $\g$ are dihedral, Ramakrishnan \cite{Ramakrishnan} proved the following cuspidality criterion:
$\f\times \g$ is cuspidal if and only if $\f$ and $\g$ cannot be induced from the same quadratic extension of $F$. 
\end{itemize}
We also recall that if $\pi$ is a dihedral cuspidal automorphic representation $\pi$ for $\GL_2(\Bbb{A}_F)$, then it can be induced from a Hecke character $\nu$ of $K$ for some quadratic extension $K$ of $F$, which will be denoted as $\pi = I_K^F(\nu)$. For such an instance, following Walji \cite[p. 4995]{Walji}, we shall say that $\pi$ has property $\mathcal{P}$ if   $\pi=I_K^F(\nu)$, and the Hecke character $\nu/\nu^{\tau}$ is invariant under  the non-trivial element $\tau$ of $\Gal(K/F)$. Property $\mathcal{P}$ characterises the decomposition of  $\Ad(\pi)$ as  follows.

\noindent ($\mathcal{P}$) If $\pi$ has property $\mathcal{P}$, then we have
\begin{equation}\label{Ad-decomp-pi1-P}
 \Ad(\pi) \simeq  \chi \boxplus \nu/\nu^{\tau} \boxplus  (\nu/\nu^{\tau})\chi,
\end{equation}
where $\chi$ is the Hecke character associated to $K/F$.

\noindent ($\mathcal{NP}$) If $\pi$  does not have property $\mathcal{P}$, then we know
\begin{equation}\label{Ad-decomp-pi2-NP}
 \Ad(\pi) \simeq  \chi  \boxplus I_{K}^F(\nu/\nu^{\tau}),
\end{equation}
where $I_{K}^F(\nu/\nu^{\tau})$ is cuspidal.

We shall prove the following theorem regarding the precise values of $\delta_{\f\times\g}$.

\begin{thm}\label{delta_fg}
Let $F$ be a number field, and let $\f$ and $\g$ be cuspidal automorphic representations for $\GL_2(\Bbb{A}_F)$ with trivial central character. Then we have
\begin{equation}\label{ff}
\delta_{\f\times\f}
=
\begin{cases}
 4 & \text{if $\f$ is dihedral and has property $\mathcal{P}$}; \\
 3 & \text{if $\f$ is dihedral but does not have property $\mathcal{P}$}; \\
 2 & \text{if $\f$ is non-dihedral}.
   \end{cases}
\end{equation}

Assume that both  $\f$ and $\g$  are dihedral and  can be induced from the same quadratic extension of $F$. If $\f$ and $\g$  are  twist-equivalent, then
\begin{equation}\label{sameQ-twist-equivalent}
\delta_{\f\times\g}
=\delta_{\f\times\f}
=
\begin{cases}
 4 & \text{if $\f$ has property $\mathcal{P}$}; \\
 3 & \text{if $\f$ does not have property $\mathcal{P}$};
   \end{cases}
\end{equation}
if  $\f$ and $\g$  are  not twist-equivalent, then
\begin{equation}
\delta_{\f\times\g}
=
\begin{cases}
 2 & \text{if both of $\f$ and $\g$ have property $\mathcal{P}$}; \\
 2 & \text{if exactly one of $\f$ and $\g$ have property $\mathcal{P}$}; \\
 2 & \text{if both $\f$ and $\g$ do not have property $\mathcal{P}$}.
   \end{cases}
\end{equation}

If $\f$ and $\g$  are dihedral representations that  cannot be induced from the same quadratic extension of $F$, then  $\delta_{\f\times\g}=1$.

If exactly one of $\f$ and $\g$ is dihedral, then $\delta_{\f\times\g}=1$. If none of $\f$ and $\g$ is dihedral, then
\begin{equation}\label{both-bob-dih}
\delta_{\f\times\g}
=
\begin{cases}
 2 & \text{if $\f$ and $\g$ are twist-equivalent}; \\
 1 & \text{otherwise}.
   \end{cases}
\end{equation}
\end{thm}

\begin{proof}
Recall that
$$
L(s,\f\times \g \times  \f\times \g )=  L(s, \Sym^2 ( \f\times \g) ) L(s, \wedge^2 ( \f\times \g) )
$$
and
$$
L(s, \Sym^2 ( \f\times \g) ) = L(s, 1 ) L(s,  \Ad (\f) \times  \Ad (\g) ).
$$

We start with noting that from (ii), it follows that  $\Ad (\f) \simeq \Ad (\g) $ if and only if $\f$ and $\g$ are twist-equivalent (this includes the case that $\f=\g$). Hence, in the case that both $\f$ and $\g$ are non-dihedral (which implies that  $\Ad (\f)$ and $\Ad (\g) $ are cuspidal),  the order of the pole of $L(s, \Sym^2 ( \f\times \g) )= L(s, 1 ) L(s,  \Ad (\f) \times  \Ad (\g) )$ at $s=1$ is 1, if  $\f$ and $\g$ are not twist-equivalent, and is 2 otherwise.
This proves the last case of \eqref{ff} and also \eqref{both-bob-dih}. I

Secondly, if exactly one of $\f$ and $\g$ is non-dihedral, by (i), \eqref{Ad-decomp-pi1-P}, and \eqref{Ad-decomp-pi2-NP}, we know that $L(s,  \Ad (\f) \times  \Ad (\g) )$ is holomorphic at $s=1$. For such an instance, $L(s, \Sym^2 ( \f\times \g) )$  has only simple pole at $s=1$, which verifies the second last assertion.

Now, it remains to consider the case that both $\f$ and $\g$ are dihedral. By (iii), if $\f$ and $\g$ cannot be induced from the same quadratic extension of $F$, then $\f\times \g$ is cuspidal, and thus $L(s,\f\times \g \times  \f\times \g )$ has only simple pole at $s=1$. This forces that $\delta_{\f\times\g}=1$. So, in the remaining part of proof, we may assume $\f$ and $\g$ can be induced from the same quadratic extension $K$ of $F$.

Firstly, if $\f$ and $\g$ are twist-equivalent (again, including the case that $\f=\g$), by (ii), $\Ad (\f) \simeq\Ad(\g)$. Furthermore, we recall that in \cite[Sec. 3.1.1]{PJW}, it was shown that for a dihedral representation $\pi$, $L(s,  \Ad (\pi) \times  \Ad (\pi) )$ has a pole of order  3  at $s=1$ if $\pi$ has property $\mathcal{P}$; otherwise, the order is 2. This, together with the decomposition $L(s, \Sym^2 ( \f\times \g) ) = L(s, 1 ) L(s,  \Ad (\f) \times  \Ad (\g) )$, proves the first two cases of \eqref{ff} as well as \eqref{sameQ-twist-equivalent}.

Finally, for the remaining cases, namely, $\f$ and $\g$ are non-twist-equivalent dihedral cuspidal representations that can be induced from the same quadratic extension $K$ of $F$, we let $\nu$ and $\psi$ denote Hecke characters of $K$ that induce $\f$ and $\g$, respectively. As before, the non-trivial element of $\Gal(K/F)$ will be denoted by  $\tau$, and $\chi$ is the Hecke character associated to $K/F$. (We recall that in this case, $\Ad (\f) \not\simeq \Ad (\g) $.)

(a) If both  $\f$ and $\g$ have property $\mathcal{P}$,  Walji \cite[Lemma 6]{Walji}  showed that $L(s,  \Ad (\f) \times  \Ad (\g) )$ has a pole of order 1 at $s=1$. Hence, $\delta_{\f\times\g}=2$. 

(b) By a dimension consideration, if exactly one of $\f$ and $\g$ has property $\mathcal{P}$, it is clear that  $L(s,  \Ad (\f) \times  \Ad (\g) )$ has a pole of order 1 at $s=1$, and thus $\delta_{\f\times\g}=2$. (Cf. \cite[Sec. 3.1.1]{PJW}.)

(c) Assume that none of  $\f$ and $\g$ has property $\mathcal{P}$. We note that as $\Ad (\g)$ is self-dual, $\overline{\chi}=\chi$ and $\overline{I_{K}^F(\psi/\psi^{\tau})}=I_{K}^F(\psi/\psi^{\tau})$. Therefore, from the fact  $\Ad (\f) \not\simeq \Ad (\g) $, it follows that $I_{K}^F(\nu/\nu^{\tau}) \not\simeq  \overline{I_{K}^F(\psi/\psi^{\tau})}$, and we conclude that $L(s,  \Ad (\f) \times  \Ad (\g) )$ has a pole of order 1 at $s=1$. This gives  $\delta_{\f\times\g}=2$ and completes the proof.
\end{proof}

\section{Preliminaries on Hilbert modular forms}

\subsection{Number field setup}

Let $F$ be a totally real number field $F$ with $[F:\mathbb{Q}]=n$, and ring of integers $\ringO_{F}$. We denote by $\d_{F}$ and $d_{F}$ the different ideal of $F$ and its norm over $\mathbb{Q}$ respectively.

We let  $\sigma_{j}:x\mapsto x_{j}$ for $j=1,\dots,n$ be the real embeddings of $F$. 
We identify an element $x\in F$ with the $n$-tuple $(x_{1},\dots,x_{n})$ in $\mathbb{R}^{n}$ where $x_j=\sigma_j(x)$. A field element $x$ is said to be totally positive (written as $x\gg0$) if $x_{j}>0$ for all $j$. For any subset $X\subset F$, we denote by $X^{+}$ the set of totally positive elements  in $X$.

To simplify the exposition of this paper, we shall frequently use the following multi-index notation.  For a given $n$-tuple $x=(x_1, \dots, x_n)$, 
we set 
\[ \norm(x)=\prod_{j=1}^n x_j \quad \text{and}\quad \trace(x)=\sum_{j=1}^n x_j.\]
Notice that if $x$ represents an element of the field $F$, the notation above agrees with the usual definitions of the field norm $\norm_{F/\Q}(x)=\prod_j \sigma_j(x)$ and trace $\trace_{F/\Q}(x)=\sum_j \sigma_j(x)$. 
We also note that certain products of the gamma function and the $J$-Bessel function are represented as follows: For given $n$-tuples $x$ and $k$, and for a scalar $a$, we have
\[ \Gamma(x+a)=\prod_{j=1}^n \Gamma(x_j+a)\quad \text{and}\quad J_{k-1}(x)=\prod_{j=1}^n J_{k_j-1}(x_j).\]

The narrow class group of $F$ is denoted by $Cl^{+}(F)$ and its cardinality by $h^{+}=h^{+}_F$. We fix a choice  $\{\a_{1},\a_{2},\dots,\a_{h^{+}}\}$ of representatives of the narrow ideal classes in $Cl^{+}(F)$. We write $\a\sim\b$ when fractional ideals $\a$ and $\b$ belong to the same narrow ideal class, in which case we have $\a=\xi\b$ for some $\xi$ in $F^{+}$. The symbol $[\a\b^{-1}]$ is used to refer to this element $\xi$ which is unique up to multiplication by totally positive units in $\ringO_{F}$.

Let $\mathbb{A}_{F}$ be the ring of ad\`eles of $F$, and let $F_v$ be the completion of $F$ at a place $v$ of $F$. For a non-archimedean place $v$, we denote by $\ringO_v$ the local ring of integers. Furthermore, we let $F_\infty=\prod_{v|\infty} F_v$, where the product is taken over all archimedean places of $F$. In what follows, we make the identifications 
$$F_{\infty}=\mathbb{R}^{n},\quad \GL_{2}^{+}(F_{\infty})=\GL_{2}^{+}(\mathbb{R})^{n},\quad\text{and}\quad  \SO_{2}(F_{\infty})=\SO_{2}(\mathbb{R})^{n}.$$
In particular, each $r\in \SO_{2}(F_{\infty})$ can be expressed as 
$$r({\theta})=\left(r(\theta_{1}),\dots,r(\theta_n)\right)=\left(\left[ \begin{array}{cc}
\cos\theta_{j} & \sin\theta_{j}  \\ 
-\sin\theta_{j} & \cos\theta_{j} \end{array} \right]\right)_{j=1}^{n}.$$ 

For an ideal $\n\subset\ringO_{F}$ and a non-archimedean place $v$ in $F$, we define the subgroup $K_{v}(\n)$ of $\GL_{2}(F_{v})$ as 
$$K_{v}(\n)=\left\{\left[ \begin{array}{cc}
a & b  \\
c & d \end{array} \right]\in \GL_{2}(\ringO_{v}): c\in \n \ringO_v\right\},$$ 
and put
\begin{equation*}
K_{0}(\n)=\prod_{v<\infty}K_{v}(\n).
\end{equation*}

\subsection{Hilbert cusp forms}

Let $k=(k_1,\dots, k_n)\in 2\mathbb{N}^{n}$, and let $\n$ be an integral ideal in $F$.   An (ad\`elic) Hilbert cusp form of weight $k$, level $\n$, and with the trivial character is a function $\f:\GL_{2}(\mathbb{A}_F)\rightarrow\mathbb{C}$ that satisfies the following conditions (\cite[Definition~3.1]{trotabas}):
\begin{enumerate}
\item The identity $\f(\gamma xgr({\theta})u)=\f(g)\exp(i\bk{\theta})$ holds for all $\gamma\in\GL_{2}(F)$, $x\in\mathbb{A}_{F}^{\times}$, $g\in\GL_{2}(\mathbb{A}_{F})$, $r(\theta)\in\SO_{2}(F_{\infty})$, and $u\in K_{0}(\n)$.
\item Viewed as a smooth function on $\GL_{2}^{+}(F_{\infty})$, $\f$ is an eigenfunction of the Casimir element $\Delta:=(\Delta_{1},\dots,\Delta_{n})$ with eigenvalue $\displaystyle{\prod_{j=1}^{n}\frac{k_j}{2}\left(1-\frac{k_j}{2}\right)}$.
\item We have the cuspidality condition: $\displaystyle{\int_{F\backslash\mathbb{A}_{F}}\f\left(\left[ \begin{array}{cc}
1 & x  \\
0 & 1 \end{array} \right]g\right)dx=0}$ for all $g\in \GL_{2}(\mathbb{A}_{F})$.
\end{enumerate} 
The space of ad\'elic Hilbert cusp forms of weight $k$, level $\n$, and with the trivial character  is denoted by $S_{k}(\n)$

Next, we introduce some notation which is needed to define the Fourier coefficients of a Hilbert cusp form $\f$. By the Iwasawa decomposition, any element $g$ in $\GL_{2}^{+}(F_{\infty})$ can be uniquely expressed as 
$$g=\left[ \begin{array}{cc}
{z} & 0  \\
0 & {z} \end{array} \right]\left[ \begin{array}{cc}
1 & {x}  \\
0 & 1 \end{array} \right]\left[ \begin{array}{cc}
{y} & 0  \\
0 & 1 \end{array} \right]r({\theta}),$$ 
with ${z}, {y}\in F_{\infty}^{+}$, ${x}\in F_{\infty}$, and $r({\theta})\in \SO_{2}(F_{\infty})$. Using this decomposition, we define $W_{\infty}^{0}(g)$ by
$$W_{\infty}^{0}(g)={y}^{\frac{k}{2}}\exp\left(2\pi i ({x}+i{y})\right)\exp\left(i\bk{\theta}\right).$$ In fact, the function $W_{\infty}^{0}$ is the new vector in the Whittaker model of the discrete series representation $\bigotimes_{j}\mathcal{D}(k_{j}-1)$ of $\GL_{2}(F_{\infty})$ (restricted to $\GL_{2}^{+}(F_{\infty})$). 

For $\f\in S_{\bk}(\n)$, $ g\in \GL_{2}^{+}(F_{\infty})$, and $\a\in I(F)$, we have the Fourier expansion 
\begin{equation*}
\f\left(g\left[ \begin{array}{cc}
\mathrm{id}(\a\d_F^{-1}) & 0  \\
0 & 1 \end{array} \right] \right)=\sum_{\nu\in(\a^{-1})^{+}}\frac{C(\nu,\a\d_F^{-1},\f)}{\norm(\nu\a)^{\frac{1}{2}}}W_{\infty}^{0}\left(\left[\begin{array}{cc}
\mathbf{\nu} & 0  \\
0 & 1 \end{array} \right] g\right),\end{equation*}
where $\mathrm{id}(\a\d_F^{-1})$ is the idele of $F$ associated with the ideal $\a\d_F^{-1}$. 
The Fourier coefficient of $\f$ at any integral ideal $\m$ in $\ringO_F$ is then defined as $C_\f(\m)=C(\nu,\a\d_F^{-1},\f)$, where $\a$ is a unique choice of representative in $\{\a_1,\dots,\a_{h^{+}}\}$ such that $\m\sim\a$, and $\nu=[\m\a^{-1}]$. Notice that $\nu$ is necessarily an element in $(\a^{-1})^+$, unique up to multiplication by $\ringO_F^{\times +}$. We say $\f$ is normalized if $C_{\f}(\ringO_{F})=1$.

As it is well-known, an ad\'elic Hilbert cusp form $\f$ in $S_{k}(\n)$ can be realized as an $h$-tuple $(f_{1},\dots,f_{h^+})$ of classical Hilbert cusp forms $f_{\nu}$ of weight $k$ with respect to the congruence subgroup
\[\Gamma_\nu(\n)=\left\{\left[\begin{array}{cc} a&b\\c&d\end{array}\right]\in\GL_2^+(F)\,:\, a, d\in\ringO_F, b\in\a_{\nu}\d_F^{-1},c\in\n\a_{\nu}^{-1}\d_F,ad-bc\in\ringO_F^{\times +}\right\},\]
where $\d_F$ is the different ideal of $F$.   For more details on this realization, the reader is referred to Garrett \cite[Ch. 1-2]{garrett}, Raghuram-Tanabe \cite[Sec. 4]{JRMS-2011}, Shimura \cite[Sec. 2]{shimura}, and Trotabas \cite[Sec.~3]{trotabas}.

The space of ad\'elic Hilbert cusp forms can be decomposed as $S_{k}(\n)=S_{k}^{\mathrm{old}}(\n)\oplus S_{k}^{\mathrm{new}}(\n)$ where $S_{k}^{\mathrm{old}}(\n)$ is the subspace of cusp forms that come from lower levels, and the new space $S_{k}^{\mathrm{new}}(\n)$ is the orthogonal complement of $S_{k}^{\mathrm{old}}(\n)$ in $S_{k}(\n)$ with respect to the Petersson inner product defined as
\begin{equation*}
\left<\f,\g\right>_{\n}=\sum_{\nu=1}^{h^+}\left<f_\nu,g_\nu\right>_{\n}=\sum_{\nu=1}^{h^+}\int_{\Gamma_\nu(\n)\backslash\h^n}
\overline{f_\nu(z)}g_\nu(z)y^kd\mu(z),
\end{equation*}
where $\h$ is the upper half-plane,  $y^k=\prod_{j=1}^n y_j^{k_j}$ and $d\mu(z)=\prod_{j=1}^ny_j^{-2}dx_{j}dy_{j}$ with $z_j=x_j+iy_j$.

The space $S_{\bk}(\n)$ has an action of Hecke operators $\{T_{\m}\}_{\m\subset\ringO_{F}}$ much like the classical setting over $\mathbb{Q}$. A Hilbert cusp form $\f$ in $S_{k}(\n)$ is said to be primitive if it is a normalized common Hecke eigenfunction in the new space. We denote by $\Pi_{k}(\n)$ the (finite) set of all primitive forms of weight $k$ and level $\n$. If $\f$ is in $\Pi_{k}(\n)$, it follows from the work of Shimura \cite{shimura} that $C_{\f}(\m)$ is equal to the Hecke eigenvalue for the Hecke operator $T_{\m}$ for all $\m\subset\ringO_{F}$. Moreover, since $\f$ is with the trivial character, the coefficients $C_{\f}(\m)$ are known to be real for all $\m$. In addition, as shown in \cite[Eq. (2.23)]{shimura}, for any primitive form $\f$, one has 
\begin{equation}\label{multi-Hecke}
C_\f(\m)C_\f(\n)=\sum_{\m+\n\subset\a}C_\f(\a^{-2}\m\n).
\end{equation}

The Ramanujan bound for Hilbert modular forms (proven by Blasius \cite{blasius}) asserts that for all $\epsilon>0$, we have
\begin{equation}\label{eqn:ramanujan-bound}
C_{\f}(\m)\ll_{\epsilon}\norm(\m)^{\epsilon}.
\end{equation}

Let $\n$ be square-free. By \cite[Proposition 2.3]{shimura}, writing $\mathfrak{l}\m=\n$, for any integral ideal $\ell\mid \mathfrak{l}$  and $f\in \Pi_{k}(\m)$, there is $\f|_{\ell}\in S_{k}(\n)$ so that $C_{\f|_{\ell}}(\ell\a)=\norm(\ell)^{\frac{1}{2}} C_{\f}(\a)$. Moreover, we have the orthogonal decomposition 
\begin{equation}\label{eqn:decomposition}
S_{k}(\n)=\bigoplus_{\mathfrak{l}\m=\n}\bigoplus_{\f\in \Pi_{k}(\m)} S_{k}(\mathfrak{l},\f),
\end{equation}
where $S_{k}(\mathfrak{l},\f)$ is the space spanned by $\f|_{\ell}$ with $\ell\mid \mathfrak{l}$. Since $\{\f|_{\ell}:\ell\mid \mathfrak{l}\}$ is linearly independent, we can construct an orthogonal basis $\{\f_{\ell}:\ell\mid \mathfrak{l}\}$ from $\{\f|_{\ell}:\ell\mid\mathfrak{l}\}$ of  $S_{k}(\mathfrak{l},\f)$ so that $\left<\f_{\ell},\f_{\ell}\right>_{\n}=\left<\f,\f\right>_{\n}$. In fact, $f_{\ell}$ is given by
\begin{align}\label{eqn:f-ell}
\f_{\ell}(g)&=\left(\frac{N(\ell)}{\rho_{\f}(\ell)}\right)^{\frac12}\sum_{\mathfrak{c}\mathfrak{d}=\ell}\frac{C_{\f}(\mathfrak{c})\mu(\mathfrak{c})}{\nu(\mathfrak{c})}N(\mathfrak{d})^{-\frac12}\f\left(g\left[\begin{array}{cc} \mathrm{id}(\mathfrak{d})^{-1}&0\\0&1\end{array}\right]\right)
\end{align}
where $\rho_{\f}(\mathfrak{c})=\prod_{\mathfrak{p}|\mathfrak{c}} (1-N(\mathfrak{p}) (\frac{C_{\f}(\mathfrak{p})}{1+N(\mathfrak{p})} )^2 )$. This construction \cite[Eq.~(11.1)]{trotabas} generalises
\cite[Eq.~(2.45)]{Iwaniec-Luo-Sarnak} to the setting of Hilbert modular forms over any totally real number field.

\subsection{Petersson trace formula}\label{sec:trace-formula}

The Petersson trace formula for Hilbert modular forms is an important ingredient in the proof of our main result. In this subsection we describe a version of this theorem that is due to Trotabas (see \cite[Theorem~ 5.5 and Proposition~6.3]{trotabas}). For conveniece, we shall first recall the definition and some properties of the $J$-Bessel function and the Kloosterman sum which appear in the statement of this theorem. The $J$-Bessel function is defined via the Mellin-Barnes integral representation as
$$J_{u}(x)=\int_{(\sigma)}\frac{\Gamma\left(\frac{u-s}{2}\right)}{\Gamma\left(\frac{u+s}{2}+1\right)}\left(\frac{x}{2}\right)^{s} ds$$
for $0<\sigma<\Re(u)$. 
It is known that for $k\geq2$ and $x>0$, we have $J_{k-1}(x)\ll\mathrm{min}(1,\frac{x}{k})k^{-\frac13}$.
In particular, if $0\leq\delta\leq1$ we have
\begin{equation}\label{J-Bessel0}
J_{u}(x)\ll \left(\frac{x}{k}\right)^{1-\delta}k^{-\frac13}\quad\quad \text{for }0\leq\delta<1. 
\end{equation} 
 As for the Kloosterman sum, it is defined as follows. Given two fractional ideals $\a$ and $\b$, let $\c$ be an ideal such that $\c^{2}\sim\a\b$. For $\nu\in\a^{-1}$, $\xi\in\b^{-1}$ and $c\in\c^{-1}\q$, the Kloosterman sum $\mathrm{Kl}(\nu,\a;\xi,\b;c,\c)$ is given by 
\[
\mathrm{Kl}(\nu,\a; \xi,\b;c,\c)=\sum_{x\in \left(\a\d_{F}^{-1}\c^{-1}/\a\d_{F}^{-1}c\right)^{\times}}\exp\left(2\pi i\trace\left(\frac{\nu x+\xi\left[\a\b\c^{-2}\right]\overline{x}}{c}\right)\right).
\]
Here, $\overline{x}$ is the unique element in $\left(\a^{-1}\d_{F}\c/\a^{-1}\d_{F}c\c^{2}\right)^{\times}$ such that $x\overline{x}\equiv 1\mod c\c$. The reader is referred to \cite[Sec. 2.2 and 6]{trotabas} for more details on this construction. The Kloosterman sum satisfies the following Weil bound: 
\begin{equation}\label{weilbound}
|\mathrm{Kl}(\alpha,\n;\beta,\m;c,\c)|\ll_{F}\norm\left(((\alpha)\n,(\beta)\m,(c)\c)\right)^{\frac{1}{2}}\tau((c)\c)\norm(c\c)^{\frac{1}{2}},\end{equation} 
where $(\a,\b,\c)$ is the g.c.d of the ideals $\a$, $\b$, $\c$, and $\tau(\n)=|\{\d\subset \ringO_{F}:\n\d^{-1}\subset\ringO_{F}\}|$ for any integral ideal $\n$. Another useful fact is the well-known estimate: for all $\epsilon>0$, we have
\begin{equation*}
\tau(\n)\ll_{\epsilon}\norm(\n)^{\epsilon}.
\end{equation*}

\begin{prop}\label{trace-formula}
For an integral ideal $\n$ in $F$ and $k\in2\mathbb{N}^n$, let $H_{k}(\n)$ be an orthogonal basis for the space $S_{k}(\n)$. Let $\a$ and $\b$ be fractional ideals in $F$. For $\alpha\in\a^{-1}$ and $\beta\in\b^{-1}$, we set 
\begin{equation}\label{eqn:petersson-sum}
\Delta_{k,\n}(\alpha\a,\beta\b)= \sum_{\f\in H_{k}(\n)}\frac{\Gamma(k-1)}{(4\pi)^{\trace(k-1)}|d_{F}|^{1/2}\left<\f,\f\right>_{\n}}C_{\f}(\alpha\a) C_{\f}(\beta\b).
\end{equation}
Then we have
\begin{align*} \Delta_{k,\n}(\alpha\a,\beta\b)
&=\1_{\alpha\a=\beta\b}
+\mathcal{C} \sum_{\substack{\mathfrak{c}^{2}\sim\a\b\\c\in\c^{-1}\n\backslash\{0\}\\\epsilon\in\mathcal{O}_{F}^{\times+}/\mathcal{O}_{F}^{\times2}}}\frac{\mathrm{Kl}(\epsilon\alpha,\a; \beta,\b;c,\c)}{\norm(c\c)}J_{k-1}\left(\frac{4\pi\sqrt{{\epsilon}{\alpha}{\beta}{\left[\mathfrak{a}\mathfrak{b}\mathfrak{c}^{-2}\right]}}}{|{c}|}\right),\end{align*}
where $\dis{\mathcal{C}=\frac{(-1)^{\trace(k/2)}(2\pi)^n}{2|d_F|^{1/2}}}$.
\end{prop}
We remind the reader that in the proposition above, we have  \[\Gamma(k-1)=\prod_{j=1}^n\Gamma(k_j-1)\] and \[J_{k-1}\left(\frac{4\pi\sqrt{{\epsilon}{\alpha}{\beta}{\left[\mathfrak{a}\mathfrak{b}\mathfrak{c}^{-2}\right]}}}{|{c}|}\right)=\prod_{j=1}^nJ_{k_j-1}\left(\sigma_j\left(\frac{4\pi\sqrt{{\epsilon}{\alpha}{\beta}{\left[\mathfrak{a}\mathfrak{b}\mathfrak{c}^{-2}\right]}}}{|{c}|}\right)\right).\] 

\subsection{Explicit formula for $L$-functions of Hilbert cusp forms}

For ${ \f\in \Pi_{k}(\n)}$, we define $L(s,\f)$ as the absolutely convergent Dirichlet series
\[L(s,\f)=\sum_{\n\subset\ringO_F}\frac{C_{\f}(\n)}{\norm(\n)^s}\] for $\Re(s)>1$. Here, the sum is over non-zero integral ideals $\n$ of $\ringO_{F}$. It is known that $L(s,\f)$ admits an analytic continuation to the entire complex plane. In the notation of Section \ref{AF}, we note that $L(s,\f)$ is an $L$-function satisfying the above formalism with
 \[
 \gamma(s,\f)=\prod_{j=1}^n\pi^{-s}\Gamma\left(\frac{s+\frac{k_j-1}{2}}{2}\right)\Gamma\left(\frac{s+\frac{k_j+1}{2}}{2}\right)
 \]
and  $\Lambda(s,\f)=(\norm(\mathfrak{n}\d_{F}))^{\frac{s}{2}}\gamma(s,\f)L(s,\f)$, where $\d_{F}$ is the different ideal of $F$. In addition, we have 
the functional equation $\Lambda(s,\f)=\epsilon(\f)\Lambda(1-s,\f)$. Notice that here we used the fact that $\f$ is self-dual (which follows from the fact that $\f$ is with trivial character).

Assuming GRH for both $L(s, \wedge^2(\f))=\zeta_{F}(s)$ and $L(s,\Sym^2 (\f))$, we have
$$
\sum_{\norm(\p)\le x} \Lambda_\f(\mathfrak{p}^2) = -x + O(x^{1/2}  \log^2( \norm(\n)\norm(k) x )).
$$
 Hence, applying Proposition \ref{prop:general-explicit-formula}, we arrive at the following proposition.
\begin{prop}\label{prop:explicit-formula}
Let $\f\in\Pi_k(\mathfrak{n})$, assume GRH for both $\zeta_{F}(s)$ and $L(s,\Sym^2 (\f))$. Let $D(\f;\phi)=\sum_{\rho}{\phi}(\frac{\gamma}{2\pi}\log R)$. Then
\begin{align*}
D(\f;\phi)&=\frac{\widehat{\phi}(0)}{\log R}\log(\norm(\mathfrak{n})\norm(k)^2)+\frac12\phi(0)\\&-\frac{2}{\log R}\sum_{\mathfrak{p}\nmid\mathfrak{n}}C_\f(\mathfrak{p})\log\norm(\mathfrak{p})\frac{\widehat{\phi}\left(\frac{\log\norm(\mathfrak{p})}{\log R}\right)}{\norm(\mathfrak{p})^{\frac{1}{2}}}+O\left(\frac{\log\log\left(\norm(\mathfrak{n})\norm(k)\right)}{\log R}\right).
\end{align*}
\end{prop}

\subsection{Explicit formula for Rankin-Selberg convolutions of Hilbert cusp forms}
In this section, we shall restrict our attention to Rankin-Selberg convolutions of Hilbert modular forms and to apply Proposition
\ref{prop:general-explicit-formula} to derive an explicit formula that will be used later.

Let ${ \f\in \Pi_{k}(\n)}$ and  ${\g\in \Pi_{k'}(\n')}$ be primitive forms, where we assume that $\n$ and $\n'$ are coprime integral ideals. 
The $L$-function for the associated Rankin-Selberg convolution is defined as \begin{equation}\label{eqn:rankin-selberg}
L(s,\f\times \g)=\zeta_{F}^{\n\n'}(2s)\sum_{\m\subset \ringO_{F}}\frac{C_{\f}(\m)C_{\g}(\m)}{\norm(\m)^{s}},
\end{equation}
where $$\zeta_{F}^{\n\n'}(2s)=\zeta_{F}(2s)\prod_{\substack{\mathfrak{l}|\n\n' \\ \mathfrak{l}\, \mathrm{prime}}}(1-\norm(\mathfrak{l})^{-2s}).$$
The series \eqref{eqn:rankin-selberg} is absolutely convergent for $\Re(s)>1$. Let
$$
\Lambda(s,\f\times \g)=\norm(\d_{F}^{2}\n\n')^{s}L_{\infty}(s,\f\times\g)L(s,\f\times \g),
$$
where 
\begin{align*}
&L_{\infty}(\f\times\g,s)\\
&=\prod_{j=1}^n(2\pi)^{-2s-\mathrm{max}\{k_{j}, k'_j\}}\Gamma\left(s+\frac{|k_j-k'_j|}{2}\right)\Gamma\left(s-1+\frac{k_j+k'_j}{2}\right)\\&=\prod_{j=1}^{n}\frac{\pi^{-2s-\max(k_j,k_j')}}{8\pi}\Gamma\Bigg(\frac{s+\frac{|k_j-k_j'|}{2}}{2}\Bigg)\Gamma\Bigg(\frac{s+1+\frac{|k_j-k_j'|}{2}}{2}\Bigg)\Gamma\Bigg(\frac{s+\frac{k_j+k_j'}{2}}{2}\Bigg)\Gamma\Bigg(\frac{s-1+\frac{k_j+k_j'}{2}}{2}\Bigg).
\end{align*} Then $\Lambda(s,\f\otimes \g)$ admits an analytic continuation to $\mathbb{C}$ as an entire function (unless $\f=\g$) and satisfies the functional equation 
$
\Lambda(s,\f\times\g)=\Lambda(1-s,\f\times \g)
$ (see \cite{prasad-ramakrishnan}).

Applying Proposition \ref{prop:general-explicit-formula} and using \eqref{eqn:pnt-f-g} yield the following result.

\begin{prop}\label{prop:fg-explicit-formula}
Let $\f\in\Pi_{k}(\n)$ and $\g\in\Pi_{k'}(\n')$ with $(\n,\n')=1$. Assuming GRH for both $L(s, \Sym^2 ( \f\times \g) )$ and $L(s, \wedge^2 ( \f\times \g) )$, we get
\begin{align*}
D(\f\times \g;\phi)
&=\frac{\widehat{\phi}(0)}{\log R}\log\Bigg( \norm(\n\n')^{2}\prod_{j=1}^{{n}}(|k_j-k_j'|+1)^2(k_j+k_j')^2\Bigg)
-\frac{\delta_{\f\times \g}}{2}\phi(0)\\
&-\frac{2}{\log R}\sum_{\mathfrak{p}\nmid \n\n'}C_{\f}(\mathfrak{p})C_{\g}(\p)\log\norm(\mathfrak{p})\frac{\widehat{\phi}\left(\frac{\log\norm(\mathfrak{p})}{\log R}\right)}{\norm(\mathfrak{p})^{\frac{1}{2}}}
+O\left(\frac{\log\log\left( \norm(kk')\norm(\n\n')\right)}{\log R}\right).
\end{align*}
\end{prop}

\section{Some technical lemmas}\label{sec:analogues-LM}
In this section, we derive some technical lemmas that allow us to express sums over primitive forms in terms of sums of Petersson type. Similar lemmas have been derived in \cite{Iwaniec-Luo-Sarnak} and adapted by \cite{liu-miller} for the setting of Hilbert modular forms of parallel weight over a totally real number field with class number one. We reproduce these lemmas in the more general setting that we require in this work. Although the proofs require minor modifications from the respective ones in \cite{Iwaniec-Luo-Sarnak}, we include them for the convenience of the reader.
\begin{lemma}\label{lem:norm}
 Let $\f\in\Pi_{k}(\m)$. Then for any  $\m\mid \n$, we have 
\begin{align*}
\left<\f,\f\right>_{\n}
&=\frac{4}{(2\pi)^n}\frac{\Gamma(k)}{(4\pi)^{\trace(k)}}\frac{\zeta_F(2)d_F^{3/2}}{R_F}\nu(\n)\prod_{\p\mid\m}\left(1-\norm(\p)^{-1}\right)Z(1,\f),
\end{align*}
where $\nu(\n)=\norm(\n)\prod_{\p\mid\n}\left(1+\norm(\p)^{-1}\right)$ and $Z(s,\f)=\displaystyle{\sum_{0\neq \q\subset\ringO_{F}}\frac{C_{\f}(\q^2)}{\norm(\q)^s}.}$
\end{lemma}
\begin{proof}
First, we recall that $\text{Res}_{s=1}L(s,\f\otimes\f)=Z(1,\f)\prod_{\p\mid\m}\left(1-\norm(\p)^{-1}\right)$. The proof now follows from \cite[Proposition~4.13, Eq. (2.27), (2.28), and (2.31)]{shimura} while noting that our definition of the Petersson inner product differs slightly from Shimura's definition. 
\end{proof}
 
\begin{lemma}\label{lem:Delta}
For $(\a,\n)=1$ and $(\b,\n)=1$, we have 
\begin{align}\label{Delta-formula-1}
 \begin{split}
\Delta_{k,\n}(\a,\b)
&=\frac{\mathcal{K}_F}{\norm(\n)\norm(k-1)}
\sum_{\mathfrak{l}\m=\n}\sum_{\f\in \Pi_{k}(\m)}\frac{Z_{\n}(1,\f)}{Z(1,\f)}C_{\f}(\a) C_{\f}(\b)\\
&=\frac{\mathcal{K}_F}{\norm(\n)\norm(k-1)}
\sum_{\mathfrak{l}\m=\n}\sum_{\ell\mid \mathfrak{l}^\infty}\frac{1}{\norm(\ell)}\Delta_{k,\m}^*(\a\ell^2,\b),
  \end{split}
\end{align}
where $\mathcal{K}_F=2^{n-2}(2\pi)^{2n} \frac{R_F}{\zeta_F(2)d_F^2}$,  $Z_{\n}(s,\f)=\displaystyle{\sum_{\q\mid \n^{\infty}}\frac{C_{\f}(\q^2)}{\norm(\q)^s}}$, and 
\begin{equation}\label{eqn:delta-star}\Delta_{k,\m}^*(\a,\b)= \displaystyle{\sum_{\f\in \Pi_{k}(\m)}\frac{Z_{\m}(1,\f)}{Z(1,\f)}C_{\f}(\a) C_{\f}(\b)} .
\end{equation}
Moreover,  we have
\begin{align} \label{delta-formula-3}
\begin{split}
\Delta^*_{k,\n}(\a,\b)
=\mathcal{K}_F^{-1}\norm(k-1)
\sum_{\mathfrak{l}\m=\n}\mu(\mathfrak{l})\norm(\m)\sum_{\ell\mid \mathfrak{l}}\frac{1}{\norm(\ell)}\Delta_{k,\m}(\a\ell^2,\b).
 \end{split}
\end{align}
\end{lemma}

\begin{proof}

We have
\begin{align*}
\Delta_{k,\n}(\a,\b)&=\sum_{\f\in H_{k}(\n)}\frac{\Gamma(k-1)}{(4\pi)^{\trace(k-1)}d_{F}^{1/2}\left<\f,\f\right>_{\n}}C_{\f}(\a) C_{\f}(\b)\\&=
\sum_{\mathfrak{l}\m=\n}\sum_{\f\in \Pi_{k}(\m)}\sum_{\ell|\mathfrak{l}}\frac{\Gamma(k-1)}{(4\pi)^{\trace(k-1)}d_{F}^{1/2}\left<\f_{\ell},\f_{\ell}\right>_{\n}}C_{\f_{\ell}}(\a) C_{\f_{\ell}}(\b)\\&=\frac{R_F}{\nu(\n)\zeta_F(2)d_F^{3/2}}\frac{2^{n-2}(2\pi)^{2n}}{N(k-1)d_{F}^{1/2}}\sum_{\mathfrak{l}\m=\n}\frac{1}{\prod_{\p\mid\m}\left(1-\norm(\p)^{-1}\right)}\sum_{\f\in \Pi_{k}(\m)}\frac{1}{Z(1,\f)}\sum_{\ell|\mathfrak{l}}C_{\f_{\ell}}(\a) C_{\f_{\ell}}(\b).
\end{align*}
Observe that \eqref{eqn:f-ell} gives
\begin{equation}\label{eqn:f-ell-coeff}
C_{\f_{\ell}}(\a)=\left(\frac{N(\ell)}{\rho_{\f}(\ell)}\right)^{\frac12}\sum_{\substack{\c\mathfrak{d}=\ell\\\mathfrak{d}|\a}}\frac{C_{\f}(\c)\mu(\c)}{\nu(\c)}C_{\f}(\a\mathfrak{d}^{-1}).
\end{equation}
Hence, we derive
\begin{align*}
\Delta_{k,\n}(\a,\b)&=\frac{\mathcal{K}_{F}}{\norm(k-1)}\sum_{\mathfrak{l}\m=\n}\frac{1}{\prod_{\p\mid\m}\left(1-\norm(\p)^{-1}\right)}\sum_{\f\in \Pi_{k}(\m)}\frac{1}{Z(1,\f)}\sum_{\ell|\mathfrak{l}}\frac{N(\ell)}{\rho_{\f}(\ell)}\\
&\times\sum_{\substack{\c_1\mathfrak{d}_1=\ell\\\mathfrak{d}_1|\a}}\frac{C_{\f}(\c_1)\mu(\c_1)}{\nu(\c_1)}C_{\f}(\a\mathfrak{d}_1^{-1}) \sum_{\substack{\c_2\mathfrak{d}_2=\ell\\\mathfrak{d}_2|\b}}\frac{C_{\f}(\c_2)\mu(\c_2)}{\nu(\c_2)}C_{\f}(\b\mathfrak{d}_2^{-1}).
\end{align*}
Since $(\a,\n)=(\b,\n)=1$, we observe that  $\mathfrak{d}_1$ and $\mathfrak{d}_2$  in the above summations are necessarily relatively prime. We write $\c_1=\b\mathfrak{d}_2$ and $\c_2=\b\mathfrak{d}_1$ to get
\begin{align*}
\Delta_{k,\n}(\a,\b)
&=\frac{\mathcal{K}_{F}}{\nu(\n)N(k-1)}\sum_{\mathfrak{l}\m=\n}\frac{1}{\prod_{\p\mid\m}\left(1-\norm(\p)^{-1}\right)}\sum_{\f\in \Pi_{k}(\m)}\frac{1}{Z(1,\f)}\sum_{\ell|\mathfrak{l}}\frac{N(\ell)}{\rho_{\f}(\ell)}\\
&\times\sum_{\substack{\b\mathfrak{d}_1\mathfrak{d}_2=\ell\\\mathfrak{d}_1|\a,\mathfrak{d}_2|\b}}\left(\frac{C_{\f}(\b)}{\nu(\b)}\right)^2\frac{\mu(\mathfrak{d}_1\mathfrak{d}_2)}{\nu(\mathfrak{d}_1\mathfrak{d}_2)}C_{\f}(\mathfrak{d}_1\mathfrak{d}_2)C_{\f}(\a\mathfrak{d}_1^{-1})C_{\f}(\b\mathfrak{d}_2^{-1})\\
&=\frac{\mathcal{K}_{F}}{\nu(\n)N(k-1)}\sum_{\mathfrak{l}\m=\n}\frac{1}{\prod_{\p\mid\m}\left(1-\norm(\p)^{-1}\right)}\sum_{\f\in \Pi_{k}(\m)}\frac{1}{Z(1,\f)}\\
&\times\sum_{\substack{\mathfrak{d}_1|\mathfrak{l},\mathfrak{d}_2|\mathfrak{l}\\\mathfrak{d}_1|\a,\mathfrak{d}_2|\b}}\frac{N(\mathfrak{d}_1\mathfrak{d}_2)}{\rho_{\f}(\mathfrak{d}_1\mathfrak{d}_2)}\frac{\mu(\mathfrak{d}_1\mathfrak{d}_2)}{\nu(\mathfrak{d}_1\mathfrak{d}_2)}C_{\f}(\mathfrak{d}_1\mathfrak{d}_2)C_{\f}(\a\mathfrak{d}_1^{-1})C_{\f}(\b\mathfrak{d}_2^{-1})\sum_{\b|\mathfrak{l}\mathfrak{d}_1^{-1}\mathfrak{d}_2^{-1}}\frac{N(\b)}{\rho_{\f}(\b)}\left(\frac{C_{\f}(\b)}{\nu(\b)}\right)^2\\&=\frac{\mathcal{K}_{F}}{\nu(\n)N(k-1)}\sum_{\mathfrak{l}\m=\n}\frac{1}{\prod_{\p\mid\m}\left(1-\norm(\p)^{-1}\right)}\sum_{\f\in \Pi_{k}(\m)}\frac{1}{\rho_{\f}(\mathfrak{l})Z(1,\f)}\\
&\times\sum_{\substack{\mathfrak{d}_1|(\mathfrak{l},\a)\\\mathfrak{d}_2|(\mathfrak{l},\b)}}N(\mathfrak{d}_1\mathfrak{d}_2)\frac{\mu(\mathfrak{d}_1\mathfrak{d}_2)}{\nu(\mathfrak{d}_1\mathfrak{d}_2)}C_{\f}(\mathfrak{d}_1\mathfrak{d}_2)C_{\f}(\a\mathfrak{d}_1^{-1})C_{\f}(\b\mathfrak{d}_2^{-1}).\end{align*}

We set \[A_{\f}(\a,\mathfrak{l})=\sum_{\mathfrak{d}_1|(\mathfrak{l},\a)}\frac{\mu(\mathfrak{d}_1)}{\nu(\mathfrak{d}_1)}N(\mathfrak{d}_1)C_{\f}(\mathfrak{d}_1)C_{\f}(\a\mathfrak{d}_1^{-1}).\] Using the multiplicativity of $C_{\f}$  given in \eqref{multi-Hecke}, we can write
\[A_{\f}(\a,\mathfrak{l})=\frac{1}{\nu((\mathfrak{l},\a))}\sum_{\mathfrak{d}_1^2|(\a,\mathfrak{l}\mathfrak{d}_1)}\mu(\mathfrak{d}_1)N(\mathfrak{d}_1)C_{\f}(\a\mathfrak{d}_1^{-2}).\]
Collecting these observations yields
 \begin{align*}
\Delta_{k,\n}(\a,\b)&=\frac{\mathcal{K}_{F}}{\nu(\n)N(k-1)}\sum_{\mathfrak{l}\m=\n}\frac{N(\m)}{\phi(\m)}\sum_{\f\in \Pi_{k}(\m)}\frac{A_{\f}(\a,\mathfrak{l})A_{\f}(\b,\mathfrak{l})}{\rho_{\f}(\mathfrak{l})Z(1,\f)}.\end{align*}
In fact, if we define
$\displaystyle Z_{\n}(s,\f)=\displaystyle{\sum_{\q\mid \n^{\infty}}\frac{C_{\f}(\q^2)}{\norm(\q)^s}}$, then one can verify that 
$\displaystyle
Z_{\n}(1,\f)=\prod_{\p}Z_\p(1,\f),
$
where
\begin{equation}\label{eqn:Zlocal}
Z_\p(1,\f)
=
\begin{cases}
 (1 +\norm(\p)^{-1})^{-1} \rho(\p)^{-1} & \text{if $\p\nmid \m$}; \\
 (1 +\norm(\p)^{-1})^{-1} (1- \norm(\p)^{-1})^{-1} & \text{if $\p\mid \m$}.
   \end{cases}
\end{equation}
It follows that for $\n$ squarefree and $\n=\m\mathfrak{l}$ we have \[Z_{\n}(1,\f)=\frac{N(\n\m)}{\nu(\n)\phi(\m)\rho_{\f}(\mathfrak{l})},\] and thus
 \begin{align*}
\Delta_{k,\n}(\a,\b)&=\frac{\mathcal{K}_{F}}{N(\n)\norm(k-1)}\sum_{\mathfrak{l}\m=\n}\sum_{\f\in \Pi_{k}(\m)}A_{\f}(\a,\mathfrak{l})A_{\f}(\b,\mathfrak{l})\frac{Z_{\n}(1,\f)}{Z(1,\f)}\\&=\frac{\mathcal{K}_{F}}{N(\n)\norm(k-1)}\sum_{\mathfrak{l}\m=\n}\sum_{\f\in \Pi_{k}(\m)}\frac{C_{\f}(\a)C_{\f}(\b)}{\nu((\a\b,\mathfrak{l}))}\frac{Z_{\n}(1,\f)}{Z(1,\f)},\end{align*}
where for the last equality we used the assumption that $(\a,\n)=(\b,\n)=1$.

The second equality in \eqref{Delta-formula-1} follows from \eqref{multi-Hecke} and \eqref{eqn:Zlocal}. The second assertion \eqref{delta-formula-3} follows from \eqref{Delta-formula-1} and the M\"obius inversion formula.
\end{proof}


%


Set
\begin{equation}\label{eqn:deltastar-def}
\Delta^\star_{k,\n}(\a) = \sum_{\f \in \Pi_{k}(\n)}  C_{\f}(\a),
\end{equation}
and observe that
\begin{equation}\label{eqn:deltastar1}
\Delta^\star_{k,\n}(\a) = \sum_{ (\m,\n)=1} \frac{1}{\norm(\m)} \Delta_{k,\n}^*(\m^2,\a).
\end{equation}
By this observation and \eqref{delta-formula-3}, we then obtain

\begin{prop}
For $(\a,\n)=1$, we have 
\begin{align} \label{delta-formula-4}
\begin{split}
\Delta^\star_{k,\n}(\a) 
&=\mathcal{K}_F^{-1}\norm(k-1)
\sum_{\mathfrak{l}\m=\n}\mu(\mathfrak{l})\norm(\m)\sum_{(\ell,\m)=1}\frac{1}{\norm(\ell)}\Delta_{k,\m}(\ell^2,\a).
 \end{split}
\end{align}
\end{prop}

In light of the work of Iwaniec-Luo-Sarnak \cite{Iwaniec-Luo-Sarnak}, we shall consider the splitting
$$
\Delta^\star_{k,\n}(\a) =\Delta'_{k,\n}(\a)+\Delta^\infty_{k,\n}(\a), 
$$
where
\begin{align} \label{def-delat'}
\begin{split}
\Delta'_{k,\n}(\a)
& =
\mathcal{K}_F^{-1}\norm(k-1)
\sum_{\substack{ \mathfrak{l}\m=\n\\ \norm(\mathfrak{l})\le X}}\mu(\mathfrak{l})\norm(\m)\sum_{\substack{(\ell,\m)=1\\ \norm(\ell)\le Y}}\frac{1}{\norm(\ell)}\Delta_{k,\m}(\ell^2,\a),
 \end{split}
\end{align}
with parameters $ X, Y \ge 1$ to be chosen later, and $\Delta^\infty_{k,\n}(\a)$ denotes the complementary sum.

We shall require the following technical lemma.

\begin{lemma}\label{lem-delta-infty-bd}
Let  $\n$ be square-free. 
Assume GRH for $L(s,\Sym^2(\f))$. Then for $(\a,\n)=1$ and any sequence $(a_\q)$ satisfying 
\begin{equation}\label{aq-cond}
\sum_{(\q,\a\n)=1}a_\q C_\f(\q)\ll (\norm(\a\n)\norm(k))^{\epsilon},
\end{equation}
we have 
\begin{align} \label{delta-infty-bd}
\sum_{(\q,\a\n)=1}a_\q \Delta^\infty_{k,\n}(\a\q)\ll \norm(k) \norm(\n) (X^{-1} +Y^{-1/2}) (\norm(\a\n) \norm(k)XY)^{\epsilon}.
\end{align}
\end{lemma}

\begin{proof}
By the definition of $\Delta^\infty_{k,\n}$, we have
\begin{align*} 
\Delta^\infty_{k,\n}(\a\q)
&=\mathcal{K}_F^{-1}\norm(k-1)\\
&\times \Bigg( 
\sum_{\substack{ \mathfrak{l}\m=\n\\ \norm(\mathfrak{l})>X}}\mu(\mathfrak{l})\norm(\m)\sum_{(\ell,\m)=1}\frac{1}{\norm(\ell)}\Delta_{k,\m}(\ell^2,\a\q)
+ \sum_{\substack{ \mathfrak{l}\m=\n\\ \norm(\mathfrak{l})\le X}}\mu(\mathfrak{l})\norm(\m)\sum_{\substack{(\ell,\m)=1\\ \norm(\ell)> Y}}\frac{1}{\norm(\ell)}\Delta_{k,\m}(\ell^2,\a\q)
 \Bigg).
\end{align*}
By the second equality of \eqref{Delta-formula-1} we have
\begin{align*} 
\Delta^\infty_{k,\n}(\a\q)
&=\sum_{\substack{ \mathfrak{l}\m=\n\\ \norm(\mathfrak{l})>X}}\mu(\mathfrak{l})\sum_{\mathfrak{L}\mathfrak{M}=\m}\sum_{(\ell,\m)=1}\frac{1}{\norm(\ell)}\sum_{\ell'\mid \mathfrak{L}^\infty}\frac{1}{\norm(\ell')}\Delta_{k,\mathfrak{M}}^*( (\ell\ell')^2,\a\q)\\
&+\sum_{\substack{ \mathfrak{l}\m=\n\\ \norm(\mathfrak{l})\le X}}\mu(\mathfrak{l})\sum_{\mathfrak{L}\mathfrak{M}=\m}\sum_{\substack{(\ell,\m)=1\\ \norm(\ell)> Y}}\frac{1}{\norm(\ell)}\sum_{\ell'\mid \mathfrak{L}^\infty}\frac{1}{\norm(\ell')}\Delta_{k,\mathfrak{M}}^*( (\ell\ell')^2,\a\q).
\end{align*}

As $\n$ is square-free, the conditions $\m\mid \n$ and $\mathfrak{L}\mathfrak{M}=\m$ imply that both $\mathfrak{L}$ and $\mathfrak{M}$ are square-free. Thus, for  $\ell'\mid \mathfrak{L}^\infty$, we know $(\ell',\mathfrak{M}) =1$. Recall that by \eqref{eqn:deltastar-def} and \eqref{eqn:deltastar1} we have
$$
\sum_{ (\ell\ell',\mathfrak{M})=1} \frac{1}{\norm(\ell\ell')} \Delta_{k,\mathfrak{M}}^*( (\ell\ell')^2,\a\q)
= \Delta^\star_{k,\mathfrak{M}}(\a\q)
=\sum_{\f \in \Pi_{k}(\mathfrak{M})}  C_{\f}(\a\q).
$$
In addition, we know that
$$
\Delta^{*}_{k,\mathfrak{M}}( (\ell\ell')^2,\a\q)=
 \sum_{\f\in \Pi_{k}(\mathfrak{M})}\frac{Z_{\mathfrak{M}}(1,\f)}{Z(1,\f)}C_{\f}((\ell\ell')^2) C_{\f}(\a\q) .
$$
Therefore, we arrive at
\begin{align*} 
\Delta^\infty_{k,\n}(\a\q)
&=\sum_{\substack{ \mathfrak{l}\mathfrak{L}\mathfrak{M}=\n\\ \norm(\mathfrak{l})>X}}\mu(\mathfrak{l})
\sum_{\f \in \Pi_{k}(\mathfrak{M})}  C_{\f}(\a\q)\\
&+\sum_{\substack{ \mathfrak{l}\m=\n\\ \norm(\mathfrak{l})\le X}}\mu(\mathfrak{l})\sum_{\mathfrak{L}\mathfrak{M}=\m}\sum_{\substack{(\ell,\m)=1\\ \norm(\ell)> Y}}\frac{1}{\norm(\ell)}\sum_{\ell'\mid \mathfrak{L}^\infty}\frac{1}{\norm(\ell')}\sum_{\f\in \Pi_{k}(\mathfrak{M})}\frac{Z_{\mathfrak{M}}(1,\f)}{Z(1,\f)}C_{\f}((\ell\ell')^2) C_{\f}(\a\q)\\&=\sum_{\substack{ \mathfrak{l}\mathfrak{L}\mathfrak{M}=\n\\ \norm(\mathfrak{l})>X}}\mu(\mathfrak{l})
\sum_{\f \in \Pi_{k}(\mathfrak{M})}  C_{\f}(\a\q)+ \sum_{\substack{ \mathfrak{l}\mathfrak{L}\mathfrak{M}=\n\\ \norm(\mathfrak{l})\le X}}\mu(\mathfrak{l})\sum_{\f\in \Pi_{k}(\mathfrak{M})}C_{\f}(\a\q)\frac{Z_{ \mathfrak{LM}}(1,\f)}{Z(1,\f)}
\sum_{\substack{(\ell,\mathfrak{LM})=1\\ \norm(\ell)> Y}}\frac{C_{\f}(\ell^2)}{\norm(\ell)},
\end{align*}
where the last equality follows from the multiplicativity of $C_{\f}$. In addition, on GRH, it can be shown that
$$
\frac{Z_{ \mathfrak{LM}}(1,\f)}{Z(1,\f)}
\sum_{\substack{(\ell,\mathfrak{LM})=1\\ \norm(\ell)> Y}}\frac{C_{\f}(\ell^2)}{\norm(\ell)}
\ll Y^{-1/2} (\norm(k)\norm( \mathfrak{L M } )Y)^{\epsilon}.
$$
In fact, it suffices to observe that
 \[\sum_{\ell\subset\ringO_F}\frac{C_{\f}(\ell^2)}{\norm(\ell)^s}=\frac{L(s,\Sym^2(\f))}{\zeta^{\mathfrak{M}}_{F}(2s)},\] and use Perron's formula along with our assumption of GRH to deduce the desired bound for the various summatory functions of the coefficients $C_{\f}(\ell^2)$.
Finally, the lemma follows from \eqref{eqn:ramanujan-bound} and \eqref{aq-cond}.
\end{proof}
\begin{prop}
We have \begin{equation}\label{eqn:size-of-newspace}
|\Pi_{k}(\n)|
=\mathcal{K}_F^{-1} \norm(\n)\norm(k-1)
\prod_{\p\mid \n} (1-\norm(\p)^{-1})
+O((\norm(k)^{\frac{13}{21}+\delta}\norm(\n)^{\frac47 +\delta})
\end{equation}
for any $0<\delta<\frac12$.
\end{prop}
\begin{proof}
As $C_\f(\ringO_{F}) =1$, we can start with writing
\begin{equation}\label{size-of-newspace1}
|\Pi_{k}(\n)| = \Delta^\star_{k,\n}(\ringO_{F}) =\Delta'_{k,\n}(\ringO_{F})+\Delta^\infty_{k,\n}(\ringO_{F})
\end{equation}
and noting that  Lemma \ref{lem-delta-infty-bd} tells us that 
\begin{equation}\label{eqn:delta-infinity-ub}\Delta^\infty_{k,\n}(\ringO_{F})
\ll (\norm(k) \norm(\n))^{1+\epsilon} (X^{-1} +Y^{-1/2}) (XY)^{\epsilon}.
\end{equation} On the other hand, from Proposition \ref{trace-formula} and \eqref{def-delat'}, it follows that
 \begin{align*} 
\begin{split}
&\Delta'_{k,\n}(\ringO_{F})\\
& =
\mathcal{K}_F^{-1}\norm(k-1)
\sum_{\substack{ \mathfrak{l}\m=\n\\ \norm(\mathfrak{l})\le X}}\mu(\mathfrak{l})\norm(\m)\sum_{\substack{(\ell,\m)=1\\ \norm(\ell)\le Y}}\frac{1}{\norm(\ell)}\Delta_{k,\m}(\ell^2,\ringO_{F})\\&=\mathcal{K}_F^{-1}\norm(k-1)
\sum_{\substack{ \mathfrak{l}\m=\n\\ \norm(\mathfrak{l})\le X}}\mu(\mathfrak{l})\norm(\m)\\
&\times \Bigg(1 + \frac{(-1)^{\trace(k/2)}(2\pi)^n}{2|d_F|^{1/2}}
\sum_{\substack{(\ell,\m)=1\\  \norm(\ell)\le Y}}\frac{1}{\norm(\ell)} 
\sum_{\substack{\mathfrak{c}^{2}\sim\ell^2 \\c\in\c^{-1}\m\backslash\{0\}\\\epsilon\in\mathcal{O}_{F}^{\times+}/\mathcal{O}_{F}^{\times2}}}
\frac{\mathrm{Kl}(\epsilon,\ell^2;1,\ringO_{F};c,\c)}{\norm(c\c)}J_{k-{1}}\left(\frac{4\pi\sqrt{{\epsilon}{\left[\ell^2\mathfrak{c}^{-2}\right]}}}{|{c}|}\right) \Bigg).
 \end{split}
\end{align*} 
Since
$$
\sum_{\substack{ \mathfrak{l}\m=\n\\ \norm(\mathfrak{l})\le X}}\mu(\mathfrak{l})\norm(\m)
= \norm(\n)
\prod_{\p\mid \n} (1-\norm(\p)^{-1}) +O\left( \frac{\norm(\n)^{1+\epsilon}}{X}\right),
$$
we get
\begin{align}\label{eqn:asymp-delta'}
\begin{split}
\Delta'_{k,\n}(\ringO_{F})
& =\mathcal{K}_F^{-1}\norm(k-1) \norm(\n)
\prod_{\p\mid \n} (1-\norm(\p)^{-1})
\\&+\mathcal{K}_F^{-1}\norm(k-1)\frac{(-1)^{\trace(k/2)}(2\pi)^n}{2|d_F|^{1/2}}\sum_{\substack{ \mathfrak{l}\m=\n\\ \norm(\mathfrak{l})\le X}}\mu(\mathfrak{l})\norm(\m)
\sum_{\substack{(\ell,\m)=1\\  \norm(\ell)\le Y}}\frac{1}{\norm(\ell)} \\
&\times
\sum_{\substack{\mathfrak{c}^{2}\sim\ell^2 \\c\in\c^{-1}\m\backslash\{0\}\\\epsilon\in\mathcal{O}_{F}^{\times+}/\mathcal{O}_{F}^{\times2}}}
\frac{\mathrm{Kl}(\epsilon,\ell^2;1,\ringO_{F};c,\c)}{\norm(c\c)}J_{k-{1}}\left(\frac{4\pi\sqrt{{\epsilon}{\left[\ell^2\mathfrak{c}^{-2}\right]}}}{|{c}|}\right) +O\left( \frac{\norm(k)\norm(\n)^{1+\epsilon}}{X}\right).
\end{split}
\end{align}
In order to bound \[J_{k-{1}}\left(\frac{4\pi\sqrt{{\epsilon}{\left[\ell^2\mathfrak{c}^{-2}\right]}}}{|{c}|}\right)=\prod_{j=1}^nJ_{k_j-1}\left(\sigma_j\left(\frac{4\pi\sqrt{{\epsilon}{\left[\ell^2\mathfrak{c}^{-2}\right]}}}{|{c}|}\right)\right),\]
 we apply  \eqref{J-Bessel0}  with $\delta_j$ being chosen as $\delta_j=0$ whenever $|\epsilon_{j}|\leq1$, and $\delta_{j}=\delta$ for some fixed $\delta\in (0,1)$ otherwise. We get
 \begin{align*} 
 \begin{split}
J_{k-{1}}\left(\frac{4\pi\sqrt{{\epsilon}{\left[\ell^2\mathfrak{c}^{-2}\right]}}}{|{c}|}\right)&\ll\prod_{j=1}^nk_{j}^{-\frac13}\prod_{j=1}^n \Bigg(\frac{\sqrt{\epsilon_j\left[\ell^2\mathfrak{c}^{-2}\right]_{j}}}{k_j |c_j|} \Bigg)
\prod_{j=1}^n \Bigg(\frac{\sqrt{\epsilon_j\left[\ell^2\mathfrak{c}^{-2}\right]_{j}}}{k_j |c_j|} \Bigg)^{-\delta_j}\\
 &\ll  \norm(k)^{\delta-\frac43}\sqrt{\norm(\ell^2)}|\norm(c\mathfrak{c})|^{\delta-1}\prod_{|\eta_j|>1}|\epsilon_j|^{-\frac{\delta}{2}}.
  \end{split}
 \end{align*}
We also use the Weil bound for Kloosterman sums \eqref{weilbound} to get
\begin{equation*}
\left|\mathrm{Kl}(\epsilon,\ell^2; 1,\mathcal{O}_F;c,\c)\right|\ll_{F}\tau(c\c)\norm(c\c)^{\frac{1}{2}}.\end{equation*} 
With these bounds in hand, we obtain
\begin{align}\label{eqn:est-J-bessel-kloosterman-sum}
\begin{split}
&\frac{(-1)^{\trace(k/2)}(2\pi)^n}{2|d_F|^{1/2}}\sum_{\substack{ \mathfrak{l}\m=\n\\ \norm(\mathfrak{l})\le X}}\mu(\mathfrak{l})\norm(\m) 
\sum_{\substack{(\ell,\m)=1\\  \norm(\ell)\le Y}}\frac{1}{\norm(\ell)} 
\sum_{\substack{\mathfrak{c}^{2}\sim\ell^2 \\c\in\c^{-1}\m\backslash\{0\}\\\epsilon\in\mathcal{O}_{F}^{\times+}/\mathcal{O}_{F}^{\times2}}}
\frac{\mathrm{Kl}(\epsilon,\ell^2;1,\ringO_{F};c,\c)}{\norm(c\c)}J_{k-{1}}\left(\frac{4\pi\sqrt{{\epsilon}{\left[\ell^2\mathfrak{c}^{-2}\right]}}}{|{c}|}\right).
\\
&\ll
 \norm(k)^{\delta-\frac43}\sum_{\substack{ \mathfrak{l}\m=\n\\ \norm(\mathfrak{l})\le X}}\norm(\m)\sum_{\substack{(\ell,\m)=1\\ \norm(\ell)\le Y}}
 \sum_{\substack{\mathfrak{c}^{2}\sim\ell^2\\c\in\c^{-1}\mathfrak{m}\backslash\{0\}}}\frac{\tau(c\c)}{\norm(c\c)^{\frac32-\delta}}\\
 &\ll
Y\norm(k)^{\delta-\frac43}\sum_{\substack{ \mathfrak{l}\m=\n\\ \norm(\mathfrak{l})\le X}}\norm(\m)^{-\frac12+\delta} \\
&\ll  YX^{\frac12-\delta}\norm(k)^{\delta-\frac43}\norm(\n)^{-\frac12+\delta+\epsilon},
\end{split}
\end{align}
where to handle the inner sum over $\mathfrak{l}$ and $\m$, we used
\begin{align*}
\sum_{\substack{ \mathfrak{l}\m=\n\\ \norm(\mathfrak{l})\le X}}\norm(\mathfrak{m})^{-\frac12+\delta}
=\sum_{\substack{ \m\mid\n\\ \norm(\m)\ge \frac{\norm(\n)}{X}}}\norm(\mathfrak{m})^{-\frac12+\delta}
&\ll \sum_{ \substack{\m\mid\n\\ \norm(\m)\ge \frac{\norm(\n)}{X}}}\left(\frac{\norm(\n)}{X}\right)^{-\frac12+\delta}\\&
\ll\norm(\n)^{-\frac12+\delta}X^{\frac12-\delta}\tau(\n)\\&\ll \norm(\n)^{-\frac12+\delta+\epsilon}X^{\frac12-\delta}
\end{align*}
for $\delta< \frac{1}{2}$.
Combining \eqref{size-of-newspace1}, \eqref{eqn:delta-infinity-ub}, \eqref{eqn:asymp-delta'}, and \eqref{eqn:est-J-bessel-kloosterman-sum}, we get

\begin{align}\label{eqn:size-of-newspace-2}
\begin{split}
|\Pi_{k}(\n)| &=\mathcal{K}_F^{-1}\norm(k-1) \norm(\n)
\prod_{\p\mid \n} (1-\norm(\p)^{-1})+O(YX^{\frac12-\delta}\norm(k)^{\delta-\frac13}\norm(\n)^{-\frac12+\delta+\epsilon})\\
&+O\left( \frac{\norm(k)\norm(\n)^{1+\epsilon}}{X}\right)
+O( (\norm(k) \norm(\n))^{1+\epsilon} (X^{-1} +Y^{-1/2}) (XY)^{\epsilon}).
\end{split}
\end{align}
Taking $Y=X^{2}$ and $X=\norm(k)^{\frac{\frac43-\delta}{\frac72-\delta}}\norm(\n)^{\frac{\frac32-\delta}{\frac72-\delta}}$ in \eqref{eqn:size-of-newspace-2}, we obtain the desired result.
\end{proof}

\section{1-level density for $L(s,\f)$: proof of Theorem \ref{thm:main1}}

The goal of this section is to obtain the desired asymptotic formula for 
\begin{equation*}
D_{k,\n}(\phi)=\sum_{\f\in\Pi_k(\n)}D(\f;\phi).
\end{equation*}

By Proposition \ref{prop:explicit-formula} and taking $R=\norm(\n)\norm(k)^2$, we have
\begin{align*}
D_{k,\n}(\phi)&=\left|\Pi_k(\n)\right|\widehat{\phi}(0)+\frac12\left|\Pi_k(\n)\right|\phi(0)\\&-\frac{2}{\log R}\sum_{\mathfrak{p}\nmid\n}\log\norm(\mathfrak{p})\frac{\widehat{\phi}\left(\frac{\log\norm(\mathfrak{p})}{\log R}\right)}{\norm(\mathfrak{p})^{\frac{1}{2}}}\Delta^\star_{k,\n}(\mathfrak{p}) +O \left(\left|\Pi_k(\n)\right|\frac{\log\log (\norm(\n)\norm(k)^2))}{\log R} \right).
\end{align*}
Observe that
\begin{align}\label{eqn:sum-over-primes-*}
\sum_{\mathfrak{p}\nmid\n}\log\norm(\mathfrak{p})\frac{\widehat{\phi}\left(\frac{\log\norm(\mathfrak{p})}{\log R}\right)}{\norm(\mathfrak{p})^{\frac{1}{2}}}\Delta^\star_{k,\n}(\mathfrak{p})
=\sum_{\mathfrak{p}\nmid\n}\log\norm(\mathfrak{p})\frac{\widehat{\phi}\left(\frac{\log\norm(\mathfrak{p})}{\log R}\right)}{\norm(\mathfrak{p})^{\frac{1}{2}}}\left(\Delta'_{k,\n}(\mathfrak{p})+\Delta^\infty_{k,\n}(\mathfrak{p})\right).
\end{align}
By Lemma \ref{lem-delta-infty-bd}, we know that
\begin{align}\label{eqn:sum-over-primes-infty}
\frac{2}{\log R}\sum_{\mathfrak{p}\nmid\n}\log\norm(\mathfrak{p})\frac{\widehat{\phi}\left(\frac{\log\norm(\mathfrak{p})}{\log R}\right)}{\norm(\mathfrak{p})^{\frac{1}{2}}}\Delta^\infty_{k,\n}(\mathfrak{p})&\ll \frac{X^{-1}+Y^{-\frac12}}{\log R}\norm(\n)\norm(k)(XY\norm(\n)\norm(k))^{\epsilon}.
\end{align}
Next, we determine the contribution of $\Delta'_{k,\n}(\mathfrak{p})$ to $D_{k,\n}(\phi)$. We have
\begin{align}
&\frac{2}{\log R}\sum_{\mathfrak{p}\nmid\n}\log\norm(\mathfrak{p})\frac{\widehat{\phi}\left(\frac{\log\norm(\mathfrak{p})}{\log R}\right)}{\norm(\mathfrak{p})^{\frac{1}{2}}}\Delta'_{k,\n}(\mathfrak{p})\nonumber\\
&=\frac{2}{\log R}\mathcal{K}_F^{-1}\norm(k-1)\sum_{\mathfrak{p}\nmid\n}\log\norm(\mathfrak{p})\frac{\widehat{\phi}\left(\frac{\log\norm(\mathfrak{p})}{\log R}\right)}{\norm(\mathfrak{p})^{\frac{1}{2}}}
\sum_{\substack{ \mathfrak{l}\m=\n\\ \norm(\mathfrak{l})\le X}}\mu(\mathfrak{l})\norm(\m)\sum_{\substack{(\ell,\m)=1\\ \norm(\ell)\le Y}}\frac{1}{\norm(\ell)}\Delta_{k,\m}(\ell^2,\mathfrak{p})\nonumber\\&=
\frac{2}{\log R}\mathcal{C}\mathcal{K}_F^{-1}\norm(k-1)\sum_{\mathfrak{p}\nmid\n}\log\norm(\mathfrak{p})\frac{\widehat{\phi}\left(\frac{\log\norm(\mathfrak{p})}{\log R}\right)}{\norm(\mathfrak{p})^{\frac{1}{2}}}
\nonumber\\
&\times\sum_{\substack{ \mathfrak{l}\m=\n\\ \norm(\mathfrak{l})\le X}}\mu(\mathfrak{l})\norm(\m)\sum_{\substack{(\ell,\m)=1\\ \norm(\ell)\le Y}}\frac{1}{\norm(\ell)}
 \sum_{\substack{\mathfrak{c}^{2}\sim\ell^2\mathfrak{p}\\c\in\c^{-1}\mathfrak{m}\backslash\{0\}\\\epsilon\in\mathcal{O}_{F}^{\times+}/\mathcal{O}_{F}^{\times2}}}\frac{\mathrm{Kl}(\epsilon,\ell^2; 1,\mathfrak{p};c,\c)}{\norm(c\c)}J_{k-{1}}
 \Bigg(\frac{4\pi\sqrt{{\epsilon}{\left[\ell^2\mathfrak{p}\mathfrak{c}^{-2}\right]}}}{|{c}|}\Bigg)\nonumber.
\end{align}
In order bound $J_{k-{1}}
 \Bigg(\frac{4\pi\sqrt{{\epsilon}{\left[\ell^2\mathfrak{p}\mathfrak{c}^{-2}\right]}}}{|{c}|}\Bigg)$,  we apply  \eqref{J-Bessel0}  with $\delta_j$ being chosen as $\delta_j=0$ whenever $|\epsilon_{j}|\leq1$, and $\delta_{j}=\delta$ for some fixed $\delta\in (0,1)$ otherwise. We get
 \begin{align}\label{J-Bessel1-bd}
 \begin{split}
J_{k-{1}}\Bigg(\frac{4\pi\sqrt{{\epsilon}{\left[\ell^2\mathfrak{p}\mathfrak{c}^{-2}\right]}}}{|{c}|}\Bigg)
&= \prod_{j=1}^nJ_{k_j-1}\left(\sigma_j\left(\frac{4\pi\sqrt{{\epsilon}{\left[\ell^2\mathfrak{p}\mathfrak{c}^{-2}\right]}}}{|{c}|}\right)\right)\\
&\ll\prod_{j=1}^nk_{j}^{-\frac13}\prod_{j=1}^n \Bigg(\frac{\sqrt{\epsilon_j\left[\ell^2\mathfrak{p}\mathfrak{c}^{-2}\right]_{j}}}{k_j |c_j|} \Bigg)
\prod_{j=1}^n \Bigg(\frac{\sqrt{\epsilon_j\left[\ell^2\mathfrak{p}\mathfrak{c}^{-2}\right]_{j}}}{k_j |c_j|} \Bigg)^{-\delta_j}\\
 &\ll \norm(k)^{\delta-\frac43}\sqrt{\norm(\ell^2\mathfrak{p})}|\norm(c\mathfrak{c})|^{\delta-1}\prod_{|\eta_j|>1}|\epsilon_j|^{-\frac{\delta}{2}}.
  \end{split}
 \end{align}
We also use the Weil bound for Kloosterman sums to get
\begin{equation}\label{Weil-bd}
\left|\mathrm{Kl}(\epsilon,\ell^2; 1,\mathfrak{p};c,\c)\right|\ll_{F}\norm\left((\ell^2,\mathfrak{p},c\c)\right)^{\frac{1}{2}}\tau(c\c)\norm(c\c)^{\frac{1}{2}},\end{equation} 
With \eqref{J-Bessel1-bd} and \eqref{Weil-bd} in hand, we obtain
\begin{align}\label{eqn:sum-over-primes-'}
\begin{split}
&\frac{2}{\log R}\sum_{\mathfrak{p}\nmid\n}\log\norm(\mathfrak{p})\frac{\widehat{\phi}\left(\frac{\log\norm(\mathfrak{p})}{\log R}\right)}{\norm(\mathfrak{p})^{\frac{1}{2}}}\Delta'_{k,\mathcal{I}}(\mathfrak{p})\\
&\ll
\frac{\norm(k)^{\delta-\frac13}}{\log R}\sum_{\substack{ \mathfrak{l}\m=\n\\ \norm(\mathfrak{l})\le X}}\norm(\m)\sum_{\substack{(\ell,\m)=1\\ \norm(\ell)\le Y}}
 \sum_{\mathfrak{p}\nmid\n}\log\norm(\mathfrak{p})\widehat{\phi}\left(\frac{\log\norm(\mathfrak{p})}{\log R}\right)
 \sum_{\substack{\mathfrak{c}^{2}\sim\ell^2\mathfrak{p}\\c\in\c^{-1}\mathfrak{m}\backslash\{0\}}}\frac{\norm\left((\ell^2,\mathfrak{p},c\c)\right)^{\frac{1}{2}}\tau(c\c)}{\norm(c\c)^{\frac32-\delta}}\\
 &\ll\frac{Y\norm(k)^{\delta-\frac13}}{\log R}\sum_{\substack{ \mathfrak{l}\m=\n\\ \norm(\mathfrak{l})\le X}}\norm(\mathfrak{m})^{-\frac12+\delta}\sum_{\mathfrak{p}\nmid\n}\log\norm(\mathfrak{p})\widehat{\phi}\left(\frac{\log\norm(\mathfrak{p})}{\log R}\right)\\
 &\ll\frac{X^{\frac12-\delta}YR^u}{\log R}\norm(\n)^{-\frac12+\delta+\epsilon}\norm(k)^{\delta-\frac13},
 \end{split}
\end{align}
provided that $\delta< \frac{1}{2}$.

Combining \eqref{eqn:sum-over-primes-*}, \eqref{eqn:sum-over-primes-infty}, and \eqref{eqn:sum-over-primes-'}, we obtain
\begin{align}\label{eqn:sum-over-primes}
\begin{split}
&\frac{2}{\log R}\sum_{\mathfrak{p}\nmid\n}\log\norm(\mathfrak{p})\frac{\widehat{\phi}\left(\frac{\log\norm(\mathfrak{p})}{\log R}\right)}{\norm(\mathfrak{p})^{\frac{1}{2}}}\Delta^\star_{k,\n}(\mathfrak{p})\\&\ll \frac{X^{\frac12-\delta}YR^u}{\log R}\norm(\n)^{-\frac12+\delta+\epsilon}\norm(k)^{\delta-\frac13}
+ \frac{X^{-1}+Y^{-\frac12}}{\log R}\norm(\n)\norm(k)(XY\norm(\n)\norm(k))^{\epsilon}.
\end{split}
\end{align}
Finally, setting $Y=X^2$ and taking $X=(\norm(\n)\norm(k))^{\eta}$ for some sufficiently small $\eta>0$, we see that
\begin{align*}
\frac{2}{\log R}\sum_{\mathfrak{p}\nmid\n}\log\norm(\mathfrak{p})\frac{\widehat{\phi}\left(\frac{\log\norm(\mathfrak{p})}{\log R}\right)}{\norm(\mathfrak{p})^{\frac{1}{2}}}\Delta^{\star}_{k,\n}(\mathfrak{p})
=o(\norm(\n)\norm(k))
\end{align*} 
provided that
\begin{equation}\label{bound-for-u}u\leq \frac{(\frac32-(\eta(\frac52-\delta)-\delta-\epsilon))\log(\norm(\n)\norm(k))-\frac16 \log\norm(k)}{\log(\norm(\n)\norm(k)^2)}.
\end{equation}

\section{1-level density for $L(s,\f\times \g)$: proof of Theorem \ref{thm:main2}}\label{1LD-R-S}

In this section, we compute the 1-level density of $L(s,\f\times\g)$. More precisely, for a fixed  $\g\in\Pi_{k'}(\n')$, we shall establish an asymptotic formula for
$$
\sum_{\f\in\Pi_{k}(\n)}D(\f\times\g;\phi).
$$
 Before we proceed with these computations, we require the asymptotic evaluation of the mean value of the parameters $\delta_{\f\times \g}$ that appear in the asymptotic formula for $D(\f\times\g;\phi)$ furnished by Proposition \ref{prop:fg-explicit-formula}.

\subsection{An averaged asymptotic of $\delta_{\f\times \g}$}\label{sec-averaged-asymp}

In this section, we aim to show that
\begin{equation}\label{averaged-asymp}
\lim_{\norm(\n)\rightarrow\infty }\frac{1}{|\Pi_{k}(\n)|} \sum_{\f\in\Pi_{k}(\n)} \delta_{\f\times \g}
=1.
\end{equation}
when the class number of $F$ is odd or $\g$ is non-dihedral.\footnote{Note that from the previous pole calculations, we know upper and lower bounds for this averaged sum.} To do so, we shall invoke a result of Ramakrishnan and Yang \cite{ramakrishnan-yang} as follows.



Suppose that $\f\in\Pi_{k}(\n)$ and $\g\in\Pi_{k'}(\n')$ with square-free $\n$ and $\n'$. Clearly, when $k\neq k'$, $\f$ and $\g$ cannot be twist-equivalent. Also, by the remark beneath \cite[Theorem A]{ramakrishnan-yang}, if $ \f \otimes \chi= \g $ for some id\` ele class character of $F$, then the conductor of $\chi$ must be $1=\mathcal{Q}_F$, and thus $\chi$ is a class group character. In our consideration, $\chi$ must be quadratic. Furthermore, recall that as the conductor of $\chi$ is $1=\mathcal{Q}_F$, the level $\f \otimes \chi$ is again $\n$. So, we conclude that
\begin{itemize}
\item if $k'\neq k$, $\g$ will never be twist-equivalent to $\f$;
\item if  $\n'\neq \n$, the number of $\g \in \Pi_{k}(\n)$ such that $\g= \f \otimes \chi$ for some (quadratic) class group character $\chi$ is bounded by the class number of $F$.
\end{itemize}
Moreover, as $\chi$ is a quadratic class group character, if the class number of $F$ is odd, then $\chi$ has to be trivial. For such an instance, there is no non-trivial $\chi$ such that $\g= \f \otimes \chi$ for any given newform $\g$ (including $\g=\f$). Consequently, none of  $\f\in\Pi_{k}(\n)$ admits a non-trivial self-twist, which forces that none of $\f$ can be dihedral.

Now, we consider the case that the class number of $F$ is odd. For a fixed  $\g\in\Pi_{k'}(\n')$, as $\norm(\n)\rightarrow\infty $, we may assume $\n\neq \n'$ so that  $\f$ and $\g$ will be non-twist-equivalent. Moreover, as the number of dihedral forms in $\Pi_{k}(\n)$ is 0, we deduce
$$
\lim_{\norm(\n)\rightarrow\infty }\frac{1}{|\Pi_{k}(\n)|} \sum_{\f\in\Pi_{k}(\n)} \delta_{\f\times \g}
= \lim_{\norm(\n)\rightarrow\infty }\frac{1}{|\Pi_{k}(\n)|} \sum_{\substack{\f\in\Pi_{k}(\n)\\ \f \text{ non-dihedral}}} \delta_{\f\times \g}.
$$
For non-dihedral $\f$, since $\f$ and $\g$ are not twist-equivalent, by Theorem \ref{delta_fg},
we always have $ \delta_{\f\times \g}=1$, and thus the desired asymptotics \eqref{averaged-asymp} follows.


We shall however remark that when the class number of $F$ is even, the above argument does not shed any light on bounding the number of dihedral forms, and thus obtaining \eqref{averaged-asymp} seems hard.  Nonetheless, for any fixed non-dihedral $\g\in\Pi_{k'}(\n')$, with $ \norm(\n')< \norm(\n) $, we always have
$$
\frac{1}{|\Pi_{k}(\n)|} \sum_{\f\in\Pi_{k}(\n)} \delta_{\f\times \g} = 1
$$
since $\f$ and $\g$ have to be non-twist-equivalent in this case. Hence, we can derive \eqref{averaged-asymp} whenever $\g$ is non-dihedral, without assuming the oddness of the class number of $F$.


\subsection{Proof of Theorem \ref{thm:main2}} From Proposition \ref{prop:fg-explicit-formula}, we know that
\begin{align*}
D(\f\times \g;\phi)
&=\frac{\widehat{\phi}(0)}{\log R}\log\Bigg( \pi^{-4n} \norm(\d_{F}^{2}\n\n')^{\frac12}\prod_{j=1}^{{2n}}(|k_j-k_j'|+1)(k_j+k_j')\Bigg)
-\frac{\delta_{\f\times \g}}{2}\phi(0)\\
&-\frac{2}{\log R}\sum_{\mathfrak{p}\nmid \n\n'}\lambda_{\f\times \g}(\mathfrak{p})\log\norm(\mathfrak{p})\frac{\widehat{\phi}\left(\frac{\log\norm(\mathfrak{p})}{\log R}\right)}{\norm(\mathfrak{p})^{\frac{1}{2}}}
+O\left(\frac{\log\log\left( \max\{c_{\f\times \g}, \norm(\n\n') \}\right)}{\log R}\right).
\end{align*}
Setting
$ \displaystyle
S(\f\times\g;\phi)=\sum_{\mathfrak{p}\nmid \n\n'}\lambda_{\f\times \g}(\mathfrak{p})\log\norm(\mathfrak{p})\frac{\widehat{\phi}\left(\frac{\log\norm(\mathfrak{p})}{\log R}\right)}{\norm(\mathfrak{p})^{\frac{1}{2}}},$
we have
\begin{align*}\sum_{\f\in\Pi_{k}(\n)}S(\f\times\g;\phi)&=\sum_{\mathfrak{p}\nmid \n\n'}\lambda_{\g}(\mathfrak{p})\log\norm(\mathfrak{p})\frac{\widehat{\phi}\left(\frac{\log\norm(\mathfrak{p})}{\log R}\right)}{\norm(\mathfrak{p})^{\frac{1}{2}}}\sum_{\f\in\Pi_{k}(\n)}\lambda_{\f}(\mathfrak{p})\\&=\sum_{\mathfrak{p}\nmid \n\n'}\lambda_{\g}(\mathfrak{p})\log\norm(\mathfrak{p})\frac{\widehat{\phi}\left(\frac{\log\norm(\mathfrak{p})}{\log R}\right)}{\norm(\mathfrak{p})^{\frac{1}{2}}}\left(\Delta'_{k,\n}(\mathfrak{p})+\Delta^\infty_{k,\n}(\mathfrak{p})\right).\end{align*}

In order to estimate the contribution of $\Delta^\infty_{k,\n}(\mathfrak{p})$, we need to apply a result similar to Lemma \ref{lem-delta-infty-bd}. We set  $a_{\p}=\lambda_{\g}(\p)\frac{\log(\norm(\p))}{\norm(\p)^{\frac12}}\widehat{\phi}\left(\frac{\log\norm(\mathfrak{p})}{\log R}\right)$ if $\p$ is  prime and $\p\nmid\n'$ and $a_{\p}=0$ otherwise. Hence,
\begin{equation*}
\sum_{\p\nmid\n\n'} \lambda_\f(\p)\lambda_{\g}(\p)\frac{\log \norm(\p)}{\norm(\p)^{\frac12}}\widehat{\phi}\left(\frac{\log\norm(\mathfrak{p})}{\log R}\right)=\sum_{\p\nmid\mathfrak{n}} \lambda_\f(\p)a_{\p}.
\end{equation*} 
Noting that $\f\neq\overline{\g}$,\footnote{As $\g$ is with trivial character, this assumption is equivalent to $\f\neq\g$.} we know that $L(s,\f\times\g)$ is entire. Thus, assuming the GRH for $L(s,\f\times\g)$ and using \cite[Theorem~5.15]{IK}, we get
$
\sum_{\p\nmid\mathfrak{n}} \lambda_\f(\p)a_{\p}\ll (\norm(\n\n')\norm(kk'))^{\epsilon}
$
for every $\f\in\Pi_{k}(\m)$ with $\m\mid\n$. Following the proof of  Lemma \ref{lem-delta-infty-bd}, we derive the following asymptotic bound
\begin{align*}
\sum_{\mathfrak{p}\nmid \n\n'}\lambda_{\g}(\mathfrak{p})\log\norm(\mathfrak{p})\frac{\widehat{\phi}\left(\frac{\log\norm(\mathfrak{p})}{\log R}\right)}{\norm(\mathfrak{p})^{\frac{1}{2}}}\Delta^\infty_{k,\n}(\mathfrak{p})
\ll \norm(\n)\norm(k)(X^{-1}+Y^{-\frac12})\left(XY\norm(\n\n')\norm(kk')\right)^{\epsilon}.
\end{align*}
It then follows that
\begin{align*}
 \frac{2}{\log R} \sum_{\f\in\Pi_{k}(\n)}S(\f\times\g;\phi)
&= \frac{2}{\log R}\sum_{\mathfrak{p}\nmid \n\n'}\lambda_{\g}(\mathfrak{p})\log\norm(\mathfrak{p})\frac{\widehat{\phi}\left(\frac{\log\norm(\mathfrak{p})}{\log R}\right)}{\norm(\mathfrak{p})^{\frac{1}{2}}}\Delta'_{k,\n}(\mathfrak{p})\\
&
+O(\norm(\n)\norm(k)(X^{-1}+Y^{-\frac12})\left(XY\norm(\n\n')\norm(kk')\right)^{\epsilon}).
\end{align*}

 As argued in the previous section, we start by writing
 \begin{align*}
&\sum_{\mathfrak{p}\nmid \n\n'}\lambda_{\g}(\mathfrak{p})\log\norm(\mathfrak{p})\frac{\widehat{\phi}\left(\frac{\log\norm(\mathfrak{p})}{\log R}\right)}{\norm(\mathfrak{p})^{\frac{1}{2}}}\Delta'_{k,\n}(\mathfrak{p})\\
&=
\mathcal{C}\mathcal{K}_F^{-1}\norm(k-1)\sum_{\mathfrak{p}\nmid \n\n'}\lambda_{\g}(\mathfrak{p})\log\norm(\mathfrak{p})\frac{\widehat{\phi}\left(\frac{\log\norm(\mathfrak{p})}{\log R}\right)}{\norm(\mathfrak{p})^{\frac{1}{2}}}
\\
&\times\sum_{\substack{ \mathfrak{l}\m= \n\\ \norm(\mathfrak{l})\le X}}\mu(\mathfrak{l})\norm(\m)\sum_{\substack{(\ell,\m)=1\\ \norm(\ell)\le Y}}\frac{1}{\norm(\ell)}
 \sum_{\substack{\mathfrak{c}^{2}\sim\ell^2\mathfrak{p}\\c\in\c^{-1}\mathfrak{m}\backslash\{0\}\\\epsilon\in\mathcal{O}_{F}^{\times+}/\mathcal{O}_{F}^{\times2}}}\frac{\mathrm{Kl}(\epsilon,\ell^2; 1,\mathfrak{p};c,\c)}{\norm(c\c)}J_{k-{1}}\Bigg(\frac{4\pi\sqrt{{\epsilon}{\left[\ell^2\mathfrak{p}\mathfrak{c}^{-2}\right]}}}{|{c}|}\Bigg),
 \end{align*}
which by \eqref{J-Bessel1-bd} and \eqref{Weil-bd}, is
  \begin{align*} 
 &\ll \sum_{\mathfrak{p}\nmid \n\n'}|\lambda_{\g}(\p)| \log\norm(\mathfrak{p})\widehat{\phi}\left(\frac{\log\norm(\mathfrak{p})}{\log R}\right)\sum_{\substack{ \mathfrak{l}\m=\n\\ \norm(\mathfrak{l})\le X}}\norm(\mathfrak{m})^{-\frac12+\delta} Y\norm(k)^{\delta-\frac13}\\
&\le  
\sum_{\mathfrak{p}\nmid \n\n'}2\log\norm(\mathfrak{p})\widehat{\phi}\left(\frac{\log\norm(\mathfrak{p})}{\log R}\right)
\norm(\n)^{-\frac12+\delta+\epsilon}X^{\frac12-\delta}
Y\norm(k)^{\delta-\frac13}
\\
&\ll X^{\frac12-\delta}YR^u\norm(\n)^{-\frac12+\delta+\epsilon}\norm(k)^{\delta-\frac13},
\end{align*}
where in the second estimate, we used the bound $|\lambda_{\g}(\p)|\le 2$.

Now, we choose  $X=Y^{\frac12}=(\norm(\n\n')\norm(kk'))^{\eta}$ for sufficiently small $\eta>0$, we get
$X^{\frac12-\delta}Y=(\norm(\n\n')\norm(kk'))^{\eta(\frac52-\frac{\delta}{2})}$. It follows that
 \begin{align*}
 \frac{2}{\log R} \sum_{\f\in\Pi_{k}(\n)}S(\f\times\g;\phi)
\ll \frac{1}{\log R} (\norm(\n\n')\norm(kk'))^{\eta(\frac52-\frac{\delta}{2})} R^u \norm(\n)^{-\frac12+\delta+\epsilon}\norm(k)^{\delta-\frac13}
+o(\norm(\n)\norm(k)).
\end{align*}
Hence, choosing $R=  \norm(\n\n')^{2}\prod_{j=1}^{{n}}(|k_j-k_j'|+1)^2(k_j+k_j')^2$, we obtain
 \begin{align*}
 \frac{2}{\log R} \sum_{\f\in\Pi_{k}(\n)}S(\f\times\g;\phi)
=o(\norm(\n)\norm(k))
\end{align*}
provided that
\[u\leq \frac{(\frac32-\delta-\epsilon)\log(\norm(\n)\norm(k))-\frac16 \log\norm(k)-\eta(\frac52- \frac{\delta}{2})\log(\norm(\n\n')\norm(kk'))}{\log\left(\norm(\n\n')^{2}\norm((|k-k'|+1)(k+k'))^2\right)}.
\]


\section{Proof of Theorem \ref{thm:main3}}

In this section, we shall work with 
$
\phi(x) = \varphi_u(x) = ( \frac{\sin (\pi ux)}{\pi ux}  )^2
$
such that $\varphi_u(0) =1$ and $\varphi_u(x)\ge 0 $. We recall that
\begin{equation*}
\widehat{\varphi_u}(t) 
=
\begin{cases}
 \frac{u-|t|}{u^2 } & \text{if $|t|\le u$}; \\
 0 & \text{otherwise},
   \end{cases}
\end{equation*}
is  supported on $(-u, u)$. 

To prove  Theorem \ref{thm:main3}(i),  we observe that 
\begin{align*}
  \sum_{m=1}^\infty m \mathcal{P}_m(\n)
\le   \frac{1}{|\Pi_{k}(\n)|} \sum_{\f \in \Pi_{k}(\n)}\sum_{\rho}{\varphi_u}(\frac{\gamma}{2\pi}\log R)
=   \frac{1}{|\Pi_{k}(\n)|}\sum_{\f \in \Pi_{k}(\n)} D(\f;\varphi_u)
,
\end{align*}
where the sum $\sum_{\rho}$ is over the zeros $\rho =\frac{1}{2}+i\gamma$ of $L(s,\f)$, counted with multiplicity. Consequently, we deduce the desired estimate \eqref{upper-bd-mPm} from Theorem  \ref{thm:main1} under GRH.


To derive Theorem \ref{thm:main3}(ii), we observe that $\mathcal{Q}_m(\n) =0$ for every odd $m$ (as the global root number in the functional equation of $L(s,\f\times\g)$ is always 1). Thus, we have
\begin{align*}
 \sum_{m=1}^\infty \mathcal{Q}_m(\n) 
  =   \sum_{n=1}^\infty  \mathcal{Q}_{2n}(\n)
 \le  \frac{1}{2}  \sum_{n=1}^\infty 2n \mathcal{Q}_{2n}(\n) 
 =  \frac{1}{2}  \sum_{m=1}^\infty m \mathcal{Q}_{m}(\n), 
\end{align*}
which yields
\begin{align*}
 \sum_{m=1}^\infty \mathcal{Q}_m(\n) 
\le \frac{1}{2}  \frac{1}{|\Pi_{k}(\n)|} \sum_{\f \in \Pi_{k}(\n)}\sum_{\rho}{\varphi_u}(\frac{\gamma}{2\pi}\log R)
=\frac{1}{2}\frac{1}{|\Pi_{k}(\n)|}  \sum_{\f \in \Pi_{k}(\n)} D(\f\times\g;\varphi_u),
\end{align*}
where the sum $\sum_{\rho}$ is over the zeros $\rho =\frac{1}{2}+i\gamma$ of $L(s,\f\times\g)$, counted with multiplicity. Hence, by Theorem \ref{thm:main2}, we obtain the claimed upper bound. On the other hand, as  $\sum_{m\ge 0}\mathcal{Q}_m(\n) =1$, under GRH, it follows from the discussion in Section \ref{1LD-R-S} that
\begin{equation*}
\mathcal{Q}_0(\n) 
=1-  \sum_{m=1}^\infty \mathcal{Q}_m(\n) 
\ge 1-   \frac{1}{2} \frac{1}{|\Pi_{k}(\n)|} \sum_{\f \in \Pi_{k}(\n)} D(\f\times\g;\varphi_u)
\end{equation*}
which is asymptotic to
$$  1- \frac{1}{2} \Bigg( \widehat{\varphi_u}(0) - \frac{1}{2 }    \varphi_u(0) \cdot  \frac{1}{|\Pi_{k}(\n)|} \sum_{\f\in\Pi_{k}(\n)}  \delta_{\f\times \g}  +o(1) \Bigg)
\ge 1 - \frac{1}{2u} +\frac{1}{4} +o(1),
$$
as $\norm(k)\norm(\n)\rightarrow \infty$, where we used the fact that $\delta_{\f\times \g}\ge 1$ (see Theorem \ref{delta_fg}).

\begin{remark} Analogously, as  $\sum_{m\ge 0}\mathcal{P}_m(\n) =1$, we have
$$
\mathcal{P}_0(\n) 
=1-  \sum_{m=1}^\infty \mathcal{P}_m(\n) 
\ge 1-   \frac{1}{|\Pi_{k}(\n)|} D_{k,\n}(\varphi_u)
\ge  1- \widehat{\varphi_u}(0)  -\frac{1}{2 }\varphi_u(0) +o(1)
= 1 - \frac{1}{u} -\frac{1}{2} +o(1)
$$
as $\norm(k)\norm(\n)\rightarrow \infty$. However, unfortunately, the non-negativity of the last term would require $u> 2$, which seems beyond the current techniques .Nevertheless, one can further split the family in accordance of the global root numbers being $1$ or $-1$ (so the corresponding symmetric type will be $\mathrm{SO(even)}$ and $\mathrm{SO(odd))}$, respectively) to establish desired non-vanishing results.
\end{remark}


\bibliographystyle{siam}
\bibliography{references}

\end{document}